\documentclass[11pt]{article}
\usepackage{algorithm}
\usepackage{algorithmic}
\usepackage{amsfonts}
\usepackage{amsmath}
\usepackage{amssymb}
\usepackage{amsthm} 
\usepackage{bbm}
\usepackage{cancel}
\usepackage{caption}
\usepackage{color}
\usepackage{comment}
\usepackage{graphicx}
\usepackage{epstopdf}
\usepackage{float}
\usepackage{hyperref}
\usepackage{cleveref}
\usepackage{fullpage}
\usepackage{lipsum}
\usepackage{mathabx}
\usepackage{mdframed}
\usepackage{subcaption}
\usepackage{wrapfig}
\usepackage{url}
\usepackage{xcolor}
\usepackage{xfrac}

\newtheorem{theorem}{Theorem}[section]
\newtheorem{proposition}[theorem]{Proposition}
\newtheorem{corollary}[theorem]{Corollary}
\newtheorem{lemma}[theorem]{Lemma}

\newtheorem{definition}[theorem]{Definition}

\theoremstyle{remark}
\newtheorem{remark}[theorem]{Remark}

\def\A{\mathbf A}
\def\B{\mathbf B}
\def\H{\mathbf H}
\def\D{\mathbf D}
\def\C{\mathbf C}
\def\K{\mathbf K}
\def\M{\mathbf M}
\def\Q{\mathbf Q}
\def\U{\mathbf U}
\def\V{\mathbf V}
\def\S{\mathbf S}
\def\I{\mathbf I}
\def\P{\mathbf P}
\def\X{\mathbf X}
\def\Y{\mathbf Y}
\def\Z{\mathbf Z}

\def\At{\widetilde{\A}}

\def\x{\mathbf x}
\def\v{\mathbf v}
\def\u{\mathbf u}
\def\s{\mathbf s}
\def\a{\mathbf a}
\def\e{\mathbf e}
\def\s{\mathbf s}
\def\b{\mathbf b}

\def\R{\mathbb R}
\def\PP{\mathbb P}
\def\E{\mathbb E}

\newcommand{\Pperp}[1]{{\P_{\!\perp #1}}}
\newcommand{\overbar}[1]{\mkern1.5mu\overline{\mkern-1.5mu#1\mkern-1.5mu}\mkern 1.5mu}

\def\EE{\mathcal E}
\def\Ec{\mathcal E}

\def\one{\mathbf 1}
\def\zero{\mathbf 0}

\def\Sigmab{\mathbf{\Sigma}}

\def\Sigmabt{\widetilde{\Sigmab}}

\DeclareMathOperator*{\argmin}{arg\,min}
\DeclareMathOperator*{\Rate}{Rate}
\DeclareMathOperator*{\Err}{Err}
\DeclareMathOperator*{\Var}{Var}
\DeclareMathOperator*{\tr}{tr}
\DeclareMathOperator*{\nnz}{nnz}

\begin{document}

\title{Sharp Analysis of  Sketch-and-Project Methods via a Connection to Randomized Singular Value Decomposition} 
\author{Micha{\l} Derezi\'nski\thanks{derezin@umich.edu, Department of Electrical Engineering and Computer Science, University of Michigan } \; and \; Elizaveta Rebrova\thanks{elre@princeton.edu, Department of Operations Research and Financial Engineering, Princeton University}} 

\maketitle

\begin{abstract}
 Sketch-and-project is a framework which unifies many known iterative methods for solving linear systems and their variants, as well as further extensions to non-linear optimization problems. It includes popular methods such as randomized Kaczmarz, coordinate descent, variants of the Newton method in convex optimization, and others. 
 In this paper, we develop a theoretical framework for obtaining sharp guarantees on the convergence rate of sketch-and-project methods. Our approach is the first to: (1) show that the convergence rate improves \emph{at least linearly} with the sketch size, and even faster when the data matrix exhibits certain spectral decays; and (2) allow for \emph{sparse} sketching matrices, which are more efficient than dense sketches and more robust than sub-sampling methods. In particular, our results explain an observed phenomenon that a radical sparsification of the sketching matrix does not affect the per iteration convergence rate of sketch-and-project. To obtain our results, we develop new non-asymptotic spectral bounds for the expected sketched projection matrix, which are of independent interest; and we establish a connection between the convergence rates of iterative sketch-and-project solvers and the approximation error of randomized singular value decomposition, which is a widely used one-shot sketching algorithm for low-rank approximation. Our experiments support the theory and demonstrate that even extremely sparse sketches exhibit the convergence properties predicted by our framework.
\end{abstract}

%=======================================

\section{Introduction}
\label{s:intro}

Randomized sketching is one of the most popular dimension-reduction techniques, applied in a variety of linear algebra, compressed sensing, and machine learning tasks  \cite{woodruff2014sketching, tropp2011structure,martinsson2020randomized,dpps-in-randnla}. An important application of sketching is to construct a random low-dimensional subspace such that much of the information contained in the data is retained after projecting onto this subspace. In particular, this arises in two algorithmic paradigms: (1) as part of an iterative solver which performs the sketching repeatedly, gradually converging to the desired solution (known as \emph{sketch-and-project}); and (2) as a basis for an approximate matrix factorization algorithm where the sketching is performed once (we refer to this as \emph{Randomized SVD}). These two paradigms have different performance metrics and traditionally have been studied independently. In this paper, we show a quantitative connection between the two and use it to obtain new convergence guarantees for iterative solvers based on sketch-and-project.

\paragraph{Sketch-and-project.} The sketch-and-project method was developed as a unified framework  for iteratively solving linear systems \cite{generalized-kaczmarz}, although it can be directly extended to non-linear optimization problems. For concreteness, let us consider our key application, projection-based linear solvers for overdetermined systems, aiming to find $\x$ that satisfies $\A\x = \b$ for $\A \in \R^{m \times n}$ and $\b \in \R^m$, where we assume that $m\geq n$ and there is a unique solution $\x_*$. Our goal is to find an $\x$ such that $\|\x-\x_*\|_{\B} := [(\x-\x_*)^T\B(\x-\x_*)]^{1/2}\leq\epsilon$, for some positive definite matrix $\B$. Instead of projecting directly onto the range of $\A$ to find $\x_* = \A^{\dagger}\b$, one can consider iterative projections on the ranges of smaller \emph{sketched} matrices $\S\A$ where $\S \in \R^{k \times m}$ with $k \ll n\leq m$. 
Then, for $t = 0, 1, \ldots$ a random sketching matrix $\S = \S(t)$ is sampled and the update rule is given~by
\begin{equation}\label{sketch-and-project}
\x_{t+1} = \argmin_{\x \in \R^n} \|\x_t - \x\|_\B^2 \quad\text{ such that }\quad \S\A\x = \S\b.
\end{equation}
This optimization problem can be solved directly and is equivalent to an iterative step
\begin{equation*}
\x_{t+1} = \x_t - \B^{-1}\A^\top\S^\top(\S\A\B^{-1}\A^\top\S^\top)^{\dagger}\S(\A\x_{t} - \b).
\end{equation*}
Here, $(\cdot)^\dagger$ denotes the Moore-Penrose pseudoinverse and the matrix $\B$ defines the projection operator. For example, $\B = \I$ defines the usual Euclidean projection. Varying the $\S$ and $\B$ matrices (i.e., the ways to sketch and to project), this general update rule reduces to randomized Gaussian pursuit, Randomized Kaczmarz, and Randomized Newton methods, among others \cite{generalized-kaczmarz, Gower2019}. 

The performance of these methods is often measured via the following worst-case expected convergence rate:
\begin{equation}\label{sp-rate}
    \Rate_\B(\A,k)\ := \
    \sup\Big\{\ \rho\ \text{ s.t. }\ \E\,\|\x_{t}-\x_*\|_\B^2\leq (1-\rho)^t\cdot\|\x_0-\x_*\|_\B^2\quad\forall\,\x_0, t \Big\}.
\end{equation}
This rate is worst-case in the sense that, while it depends on the problem and sketching matrix, it does not depend on $\x_0 \in \mathbb{R}^n$ or $t = 1, 2, \ldots$, because of taking the supremum. While implicit ways to estimate this quantity, as a function of the distribution of the sketches, were known from early papers in the area \cite{gower2017randomized}, the practical goal is to give explicit dependence  on the sketch size and spectral characteristics of the data. Such explicit characterizations have remained elusive even in such simple cases as Gaussian sketching matrices, and are virtually non-existent for more practical sparse and structured sketching methods.

\paragraph{Contribution 1: Sharp convergence rates and connections to the data spectral decay.}
We develop a new framework for obtaining sharp guarantees on the convergence rate of sketch-and-project methods.
At a high level, our main theoretical results (Theorems \ref{t:gaussian}, \ref{t:sub-gaussian}~and~\ref{c:less}) show that if the sketching matrix $\S$ exhibits a sufficiently Gaussian-like distribution, then:
\begin{equation}
    \Rate_\B(\A,k) \ \gtrsim\ \frac{k\,\sigma_{\min}^2(\At)}{\Err(\At,k-1)}
    \geq \frac{k\,\sigma_{\min}^2(\At)}{\|\At\|_F^2}
    \qquad\text{for}\qquad \At=\A\B^{-1/2},\label{eq:connection}
\end{equation}
where $\sigma_{\min}(\cdot)$ denotes the smallest singular value and $\Err(\A,k) := \E\,\|\A(\I- (\S\A)^\dagger\S\A)\|_F^2$ is the expected Frobenius norm error of projecting $\A$ onto the subspace defined by $\S\A$, which 
is a standard error metric for rank-$k$ Randomized SVD \cite{tropp2011structure}. Thus, our results explicitly connect two disjoint domains where matrix sketching is a standard algorithmic tool.

To gain intuition, consider the case when $\B=\I$ (in which case $\At=\A$) and note that $\Err(\A,k)$ is a non-increasing function in $k$, so that $\Err(\A,k)\leq \Err(\A,0)=\|\A\|_F^2$. When $k=1$, \eqref{eq:connection} recovers the standard convergence guarantee for Randomized Kaczmarz \cite{strohmer2009randomized}, i.e., $\Rate(\A,1)\gtrsim \sigma_{\min}^2(\A)/\|\A\|_F^2$, but when we start increasing the sketch size~$k$, our convergence rate bound improves at least linearly with $k$. For the so-called Block Gaussian Kaczmarz, i.e., where $\S$ is Gaussian, this is a direct improvement over prior convergence guarantees \cite{rebrova2021block}, and the first result to demonstrate linear improvement of $\Rate_\B(\A,k)$ as a function of $k$ over all sketch sizes and input matrices (first part of Corollary~\ref{cor:poly-decays}).

In addition, we show that if the Randomized SVD error decreases rapidly with $k$ (i.e., $\Err(\A,k)\ll\|\A\|_F^2$), which occurs when $\A$ exhibits fast spectral decay, then the convergence rate of sketch-and-project improves \emph{superlinearly} with $k$. To our knowledge, this is the first result of this kind. For example, in combination with existing bounds for Randomized SVD \cite{tropp2011structure}, we show that, when the input matrix $\A$ exhibits polynomial spectral decay with $\sigma_i^2(\A)\asymp i^{-\beta}$ for $\beta >1$, then Block Gaussian Kaczmarz exhibits convergence with $\Rate(\A,k)\gtrsim k^\beta\sigma_{\min}^2(\A)/\|\A\|_F^2$ (second part of Corollary~\ref{cor:poly-decays}).

Furthermore, we show that for certain matrices $\A$, the convergence of Block Gaussian Kaczmarz can become entirely independent of the condition number of $\A$ for the right choice of sketch size $k$. Namely, in Corollary \ref{c:flat-tailed} we show that if the spectrum of $\A$ exhibits a sharp decay followed by a flat tail (this can be caused either by noise in the data or by deliberate regularization of the problem), then choosing sketch size $k$ proportional to the number of singular values in the first part of the spectrum leads to the convergence with $\Rate(\A,k)\gtrsim k/n$. Note that, unlike all previously known guarantees for sketch-and-project, this rate does not depend on the condition number $\kappa(\A)=\sigma_{\max}(\A)/\sigma_{\min}(\A)$, which can be arbitrarily large for such matrices. 

All of the above phenomena describing the dependence of the convergence rate on the sketch size are verified by our empirical results. While practical choice of sketch size is a decision depending on multiple factors, including storage and computation-communication trade-offs, the new convergence guarantees significantly advance our understanding of the quantitative effect that the sketch size has on the convergence. Depending on the information available on the spectrum of $\A$, equipped with the new rates, we might have more or less incentives to increase the sketch size further, and get an estimate on the minimum viable sketch size (and memory requirements) for a particular application.

\paragraph{Contribution 2: Convergence rate analysis for sparse and structured sketches.}
Prototypical examples of sketching matrices are Gaussian matrices with independent entries and other matrices satisfying the Johnson-Lindenstrauss property (preserving the geometry of the data), or block-identity matrices that effectively subsample the data matrix. A variety of other sketching matrices have been proposed \cite{ailon2009fast,cw-sparse,nn-sparse,less-embeddings}, exploiting sparse and/or structured matrix representations to optimize the trade-off between the complexity of the model, its storage requirements, its effectiveness in compressing the data, and the cost of sketching. While the performance of these optimized sketching methods is well-understood for many tasks (including Randomized SVD), the convergence analysis of sketch-and-project is an exception. Here, prior work has been limited to Gaussian \cite{rebrova2021block} and certain block-identity sketching matrices \cite{needell2014paved} (or to results that do not provide an explicit bound on the convergence rate).

We overcome these limitations by providing a framework for analyzing sketch-and-project under any sketching distribution that satisfies certain generic concentration assumptions (Theorem~\ref{t:sub-gaussian}). This includes not only Gaussian matrices, but also matrices with i.i.d.~sub-gaussian entries (e.g., random $\pm1$ entries, which are cheaper to generate than Gaussians), as well as certain sparse sketching matrices. Then, we show that our convergence guarantees for sketch-and-project can be shown when using a family of sparse sketching matrices called Leverage Score Sparsified (LESS) embeddings \cite{less-embeddings} (Theorem~\ref{c:less}).  LESS embeddings use a carefully constructed sparsity pattern, based on the leverage scores of the input matrix $\A$ \cite{fast-leverage-scores}, to produce a sketching matrix that mimics the properties of a Gaussian sketch while enjoying the efficiency of sparse matrix multiplication. We also note that the same effect can be achieved by simpler uniformly sparsified sketching matrices if the input matrix is preconditioned by a randomized Hadamard transform (see Section \ref{s:main-sparse}). 

It was observed before \cite{rebrova2021block} that sparse sketches exhibit very similar per iteration convergence as dense Gaussian sketches, while being much faster in practice (indeed, every iteration is significantly quicker if it avoids multiplication with a dense Gaussian matrix). Thus, the analysis of non-Gaussian, and in particular, sparse and discrete sketches is necessary to give convergence guarantees for more practical solvers.  To our knowledge, we give the first non-trivial theoretical result for sketch-and-project with sparse sketching matrices. We also observe empirically that, for many real-world problems, the Gaussian-like performance of sketch-and-project holds even for extremely sparse sketching matrices and even without~preconditioning.

\paragraph{Extensions to nonlinear optimization.}

While, for the sake of simplicity, our main results are focused on solving linear systems, they can be naturally extended to a number of nonlinear optimization algorithms based on the sketch-and-project framework \cite{agmon1954relaxation, leventhal2010randomized, briskman2015block,de2017sampling}. We illustrate this in Section~\ref{s:main-newton}, showing that improved convergence guarantees can be obtained for a stochastic Newton method known as Randomized Subspace Newton \cite{Gower2019}. Here, sketching is applied to the Newton system that arises at each step of the optimization, and the convergence is characterized via the spectral properties of the Hessian of the objective.

\subsection{Our approach: New spectral analysis of expected projection matrix}

Our main technical contribution is a sharp spectral analysis of the expectation of the projection matrix $\P$ corresponding to the random subspace defined by the sketch $\S\A$, addressing a long-standing challenge in the analysis of sketch-and-project. 
The expected projection matrix $\E[\P]$ arises in the analysis through the following well-known characterization of the sketch-and-project worst-case convergence rate in terms of its smallest eigenvalue, given by \cite{generalized-kaczmarz}:
\begin{equation}
    \Rate_\B(\A,k) \ =\  \lambda_{\min}(\E[\P])\qquad\text{for}\qquad\P := \At^\top\S^\top(\S\At\At^\top\S^\top)^\dagger\S\At,\label{eq:rate-lambda}
\end{equation}
where $\At=\A\B^{-1/2}$, see more details below in \eqref{eq:characterization}. Thus, lower bounding the smallest eigenvalue of $\E[\P]$ immediately implies a convergence guarantee for sketch-and-project. Yet, spectral analysis of the expected projection matrix has proven challenging for most sketching distributions, even such standard ones as the Gaussian sketch. We address this challenge, going even beyond the smallest eigenvalue, for Gaussian, sub-gaussian and sparse sketches.

The key technical part of our analysis is a careful decomposition of the projection matrix $\P$ via rank-one update formulas, and utilizing concentration of random quadratic forms based on the Hanson-Wright inequality. The connection to Randomized SVD error comes from the fact that this error can also be expressed as a function of the expected projection matrix, namely, $\Err(\At,k) = \tr\At^\top\At(\I-\E[\P])$, and this function naturally arises in our analysis. We remark that our proof in the Gaussian case (Theorem \ref{t:gaussian}) relies on a completely different decomposition of $\P$ than our full-spectrum analysis for general sub-gaussian sketches (Theorem~\ref{t:sub-gaussian}). This is why the former applies to a broader range of sketch sizes and input matrices, whereas the latter applies to a wider family of sketching distributions and also yields tight two-sided estimates for the larger eigenvalues of $\E[\P]$. Finally, the sharpness of our resulting estimates for $\Rate_\B(\A,k)$ is also supported by the experiments, see Figure~\ref{fig-surrogates}. We note that the actual convergence rate can and will be somewhat better than the worst-case convergence rate $\Rate_\B(\A,k)$ due to some iterates  not being in the worst-case position.

\subsection{Paper organization and roadmap} The paper is organized as follows. 
We start by providing additional background and related work in Section \ref{s:related}. Then, in Section~\ref{s:main-results}, we give formal statements and additional discussion of all our main results: 
In \ref{s:main-gaussian}, we give a lower bound on $\lambda_{\min}(\E[\P])$ under the assumption that the sketching matrix has Gaussian entries, and translate it to convergence rate guarantees for sketch-and-project for data with various spectral decays. In~\ref{s:main-sub-gaussian}, we give full-spectrum two-sided estimates for $\E[\P]$ for a range of sketching distributions that satisfy sub-gaussian concentration. They lead to sharper convergence rate estimates when the sketch size $k$ is smaller than the stable rank of the matrix $\A$. In~\ref{s:main-sparse}, we show that certain very \emph{sparse sketching matrices} (LESS embeddings), are covered by our full-spectrum analysis of the expected projection. Finally, in~\ref{s:main-newton}, we show the extensions to certain stochastic Newton methods in convex optimization.
Section~\ref{s:all-proofs} contains all main proofs of the theorems listed above. Empirical evidence is given in Section~\ref{s:experiments}. The appendix contains auxiliary lemmas and additional experimental data.

\subsection{Notations} We use $\a_i^\top$ and $\sigma_i(\A)$ to denote the $i$th row and $i$th largest singular value of matrix $\A$, while $\A^\dagger$ denotes the Moore-Penrose pseudoinverse. Also, $\|\A\|$, $\|\A\|_F$ and $\|\A\|_*$ are the spectral/operator, Frobenius/Hilbert-Schmidt and trace/nuclear norms, respectively. For positive semidefinite matrices $\B$ and $\C$, we use $\B\preceq\C$ to denote the Loewner ordering, and $\lambda_i(\B)$ is the $i$th largest eigenvalue. We let $\one_{\EE}$ be the indicator function of an event $\EE$, and $\neg \EE$ for its complement event.
We use $\Rate(\A,k)$ as a short-hand for  \eqref{sp-rate} when $\B=\I$. We let $C, C', C_1, c$ denote absolute constants whose values might differ from line to line.

%=======================================

\section{Background and related work}\label{s:related}
In this section, we put our results in perspective by discussing some relevant prior work, including a brief introduction to Randomized SVD, as well as the Kaczmarz algorithm and its extensions.

\subsection{Randomized singular value decomposition}\label{seq:randSVD}
This is a family of methods which use a randomized subspace generated from a sketch of the input matrix to accelerate performing restricted versions of matrix factorizations such as QR and SVD \cite{tropp2011structure}, as follows:
\begin{enumerate}
    \item First,  generate a sketch $\S\A$ of size $k\times n$, letting $\Q$ be the orthonormal basis for the row-span of $\S\A$, so that $\A\approx\A\Q\Q^\top$.
    \item Next, use $\Q$ to help compute a factorization of $\A$. E.g., we can compute an SVD of the small matrix $\A\Q$, letting $\A\Q= \U\Sigmab\widetilde\V^\top$, and then set $\V=\Q\widetilde\V$, so that $\A\approx\U\Sigmab\V^\top$. 
\end{enumerate} 
When the sketch is based on a block identity matrix, then this approach is also known as Column/Row Subset Selection  \cite{BoutsidisMD08}.
It is important that the number of columns in $\Q$ (i.e., the sketch size $k$) is small, as that reduces the computational cost of the second step. While there are many factorization algorithms that can be applied to this paradigm, the effectiveness of this approach is limited by the approximation error of the random orthonormal basis $\Q$. One common way to measure the error is via the expected Frobenius norm\footnote{The spectral norm is also standard, but the Frobenius norm best suits our analysis of sketch-and-project.} of the residual after projecting onto the subspace defined by the basis (e.g., see \cite{BoutsidisMD08,tropp2011structure}):
\begin{equation}
    \Err(\A,k) \ :=\ \E\,\|\A-\A\Q\Q^\top\|_F^2 = \E\,\|\A(\I-(\S\A)^\dagger\S\A)\|_F^2,\label{svd-error}
\end{equation}
where $\Q\Q^\top=\A^\top\S^\top(\S\A\A^\top\S^\top)^\dagger\S\A =(\S\A)^\dagger\S\A =\P$ is the sketched projection matrix which is central to our analysis of sketch-and-project.

\subsection{Sketching for linear systems}
\label{s:related-linear}
One of the most popular variants of the sketch-and-project method is the Kaczmarz algorithm for solving linear systems \cite{kaczmarz1937angenaherte}, which solves the system $\A \x = \b$ by sequential projections of the iterates onto the individual equations. 
The method is especially efficient for solving tall highly overdetermined linear systems, and in streaming applications. Its convergence guarantees of the form \eqref{sp-rate} were first proved in \cite{strohmer2009randomized} with $\Rate \geq \frac{\sigma_{\min}^2(\A)}{\|\A\|_F^2}$ under a randomized sampling rule (sampling proportionally to the row-norms of $\A$), which is exactly equivalent to sketch-and-project update \eqref{sketch-and-project} with $\B = \I$ and $\S$ being a random standard basis vector in $\R^n$. 

Variants of the Kaczmarz method that exploit several rows at each iteration, often referred to as block methods, have been extensively studied. In \cite{needell2014paved}, it is shown that 
$\Rate \geq c\frac{\sigma_{\min}^2(\A)}{\log(n)\|\A\|^2}$
can be guaranteed for a particular block partition of the matrix $\A$. Empirically, block Kaczmarz method significantly outperforms randomized Kaczmarz method but the extent of this advantage, depending on $k$ and $\A$, remained somewhat mysterious. In the special case of sketch size $k = 2$, the rate acceleration was connected to the row coherence of the matrix $\A$ \cite{needell2013two}. In \cite{haddock2021greed}, the rate of block Kaczmarz method was bounded by the rate of Sampling Kaczmarz-Motzkin method, which does not require a pre-determined block size or a fixed block partition, but has slower iterations and no explicit dependence of the theoretical rate on the spectrum of $\A$.

In \cite{rebrova2021block}, the authors show a nearly-linear growth of the convergence rate with the sketch size, for continuous sketching distributions. Namely, it is proved that for some constant $C>0$,
$$
\Rate(\A,k)\geq \frac{k\sigma_{\min}^2(\A)}{C(\sqrt k \|\A\| + \|\A\|_F)^2}, 
$$
under the assumption that $\S$ is a $k \times m$ matrix with independent standard Gaussian entries. Note that this establishes linear growth only for $k$ that is smaller than the stable rank of the matrix $\A$. 
Compare this to our Corollary \ref{cor:poly-decays}, where we show $\Rate(\A,k)\geq ck\sigma_{\min}^2(\A)/\|\A\|_F^2$ for arbitrary matrices, and even better rates under certain spectral decay assumptions (Corollaries \ref{cor:poly-decays} and \ref{c:flat-tailed}).
Although the analysis in \cite{rebrova2021block} can be slightly extended towards more general sketching distributions, one of the key parts of the proof is a random matrix deviation inequality resulting in the unavoidable mixture of the spectral and Frobenius norm in the denominator. 
In this work, we propose a new approach for the analysis of the random sketched projection matrix, based on rank one update formulas rather than on splitting a matrix into a product of two matrices followed by~decoupling. 

\subsection{Spectral analysis of expected projections}Full-spectrum analysis of the expected projection is not available for most block sampling distributions (i.e., those that select a random subset of rows of $\A$). 

One notable exception is sketches based on Determinantal Point Processes (DPPs, \cite{dpps-in-randnla}). These are distributions over subsets of rows, where not only the individual row probabilities, but also their pairwise correlations, are carefully chosen based on the matrix $\A$. The resulting sketch is therefore more expensive to construct, although practical algorithms are available in some settings \cite{dpp-intermediate,alpha-dpp}. In this context, \cite{randomized-newton} and \cite{rodomanov2020randomized} gave exact characterizations of the spectrum of $\E[\P]$ under certain DPP distributions, and they used these characterizations to analyze the convergence of a coordinate descent method based on sketch-and-project. These results were later extended and adapted to Randomized SVD-type error bounds \cite{nystrom-multiple-descent} and interpolated regression analysis \cite{surrogate-design}. This line of works relies on proof techniques which are unique to DPP-based sketches, and therefore it is much less broadly applicable, compared to our unified analysis of Gaussian, sub-gaussian and sparse sketching~matrices.

In the context of sub-gaussian sketches, full-spectrum analysis of the expected projection matrix was initiated by \cite{precise-expressions}.
Their primary focus is on Randomized SVD and they study the
full-spectrum analysis of the expected \emph{residual} projection, i.e.,
$\E[\I-\P]$. Note that in the context of
sketch-and-project the goal is to lower-bound the smallest
eigenvalue of $\E[\P]$. A multiplicative upper
bound on the largest eigenvalue of the residual projection, i.e.,
$\lambda_{\max}(\E[\I-\P])\leq(1+\epsilon)\cdot (1-\rho)$ implies a much weaker \emph{additive} bound on the convergence rate, i.e.,
$\lambda_{\min}(\E[\P])\geq \rho - \epsilon$. 
Since $\rho$ scales with
the smallest singular value of $\A$, this means that the worst-case convergence rate
bound derived from their result is vacuous unless $\epsilon< \rho$ and the
condition number of $\A$ is close to 1. In this work, we perform the full-spectrum analysis directly on the expected projection, rather than its residual, and as a result, are able to obtain a \emph{multiplicative} bound on the smallest eigenvalue, i.e., $\lambda_{\min}(\E[\P])\geq (1-\epsilon)\cdot\rho$, which is non-vacuous as long as $\epsilon<1$, regardless of the value of $\rho$. To achieve this, in our analysis (Theorem \ref{t:sub-gaussian}) we must overcome a key obstacle that does not arise in \cite{precise-expressions}: the rank-one decomposition of the expected projection matrix yields potentially non-symmetric and highly ill-conditioned random matrices. We address this by a careful symmetrization argument combined with a self-recursive bound.

\subsection{Sketch-and-project beyond linear systems}
 We treated linear solvers as our main application example mainly because it is a useful showcase of the phenomena appearing in sketch-and-project methods, related to the sketch sizes and sketching distributions, as well as the spectral properties of the data matrix. However, our analysis is applicable to any instance of sketch-and-project. A direct extension of solving overdetermined linear systems is feasibility questions for linear and convex feasibility problems, studied in the sketch-and-project viewpoint in \cite{necoara2021randomized}. Special cases of projection-based algorithms for linear feasibility, as well as their block variants, were studied earlier in \cite{agmon1954relaxation, leventhal2010randomized, briskman2015block,de2017sampling}.
 
 Further, numerous extensions of sketch-and-project have been aimed at efficiently minimizing convex functions \cite{Gower2019,hanzely2020stochastic}, solving nonlinear equations \cite{sketched-newton-raphson} and inverting matrices \cite{gower2017randomized}. In Section \ref{s:main-newton}, we explain how our convergence results can be extended to these settings using the example of minimizing convex functions with stochastic Newton methods \cite{Gower2019,hanzely2020stochastic}. This is possible whenever the convergence analysis relies on the smallest eigenvalue of the expected sketched projection matrix, sometimes referred to as the \emph{stochastic condition number} in this line of works. Other examples where this setup occurs include the Sketched Newton-Raphson algorithm \cite{sketched-newton-raphson} and the Stochastic Iterative Matrix Inversion algorithm \cite{gower2017randomized}.
\section{Main results}\label{s:main-results}
In this section, we formulate our main theoretical results. Here and further, without loss of generality, we let $\B=\I$ in the formulation  \eqref{sketch-and-project} of sketch-and-project (to revert this, it suffices to replace $\A$ with $\A\B^{-1/2}$). The sketched projection matrix is then:
\begin{equation*}
    \P := \A^\top\S^\top(\S\A\A^\top\S^\top)^\dagger\S\A = (\S\A)^\dagger\S\A,
\end{equation*}
where $\S$ is $k\times m$. The iteration step of sketch-and-project reduces to the following update:
\begin{equation}\label{sp-iteration}
\x_{t+1} = \x_t - (\S\A)^\dagger\S(\A\x_{t}-\b) = \x_t- \P(\x_t-\x_*).
\end{equation}
Relying on the characterization given by \cite{generalized-kaczmarz}, the above update satisfies:
\begin{equation}
    \E\,\|\x_{t+1}-\x_*\|^2 = \|\x_t-\x_*\|_{\E[\I-\P]}^2\leq(1-\lambda_{\min}(\E[\P]))\cdot \E\,\|\x_t-\x_*\|^2\label{eq:characterization}   
\end{equation}
and the inequality becomes an equality when $\x_t-\x_*$ is the eigenvector of $\E[\P]$ associated with its smallest eigenvalue. By the definition in \eqref{sp-rate}, this means that $\Rate(\A,k)= \lambda_{\min}(\E[\P])$,  so in the remainder of this section we focus primarily on lower-bounding $\lambda_{\min}(\E[\P])$. 

\subsection{Smallest eigenvalue analysis with Gaussian sketches}
\label{s:main-gaussian}
The first part of our results gives lower bounds for the spectrum of the sketched projection matrix in the case of standard Gaussian sketches. The result of Theorem~\ref{t:gaussian} holds for any matrix $\A \in \R^{m \times n}$ and all sketch sizes up to $n/4 - o(n)$, see the details in Remark~\ref{r:sketch-sizes}. It shows that the convergence rate improves at least linearly with respect to the sketch size $k$. Additionally, in Corollary~\ref{cor:poly-decays},  we show faster rate increases with respect to $k$ for matrices with fast spectral decays, and in Corollary~\ref{c:flat-tailed}, we show that for certain matrices, the rate can become independent of the condition number with the right choice of $k$.  
The proof of Theorem~\ref{t:gaussian} is in Section~\ref{s:proofs-gaus}.

Here and throughout, for a random $k\times m$ sketching matrix $\S$, we will denote the projection onto the span of $\S\A$, along with its Randomized SVD error measured via Frobenius norm, as:
\begin{equation}\label{randsvd_error}
	\P:=(\S\A)^\dagger\S\A\qquad\text{and}\qquad\Err(\A,k) := \E\,\|\A(\I-\P)\|_F^2. 
\end{equation}

  \begin{theorem}\label{t:gaussian}
    Let $\A$ be an $m\times n$ matrix with rank $n$ and let $\S$ be a $k\times m$ Gaussian matrix. 
    Then, the smallest eigenvalue of $\E[\P]$ satisfies:
\begin{equation} \label{thm1_claim}
      \lambda_{\min}(\E[\P]) \geq (1-\epsilon) \frac{k\sigma_{\min}^2(\A)}{\Err(\A,k-1)}
      \qquad\text{where}\qquad
      \epsilon \leq 
      \frac{4k}n + \frac{8\log(3n)}n.
    \end{equation}
    \end{theorem}

Note that the right hand side uses $\Err(\A,k-1)$, i.e., Randomized SVD error based on a Gaussian matrix of size $k-1$, even though the left hand side is defined using a sketch of size $k$. In particular, for $k=1$ we get $\Err(\A,0)$, which simplifies to $\|\A\|_F^2$.
    \begin{remark}[Sketch size and dependence on $\epsilon$] \label{r:sketch-sizes} 
    It can be verified that for any $n\geq 250$ and $k\leq n/5$ we have $\epsilon\leq 0.9$, whereas for any $n\geq 1000$ and $k\leq n/20$ we have $\epsilon\leq 0.25$ (see also Remark~\ref{r:gaussian-variant}).  
    If $k$ is fixed and $n \to \infty$, then $\epsilon \to 0$. Thus, since $\Err(\A,k-1) \le \|\A\|_F^2$ for any $k\geq 1$, Theorem~\ref{t:gaussian} immediately gives a simple guarantee for sketch-and-project:
   \begin{equation}\label{1-s} 
   \Rate(\A,k) \geq (1-\epsilon_{n,k}) \frac{k\sigma_{\min}^2(\A)}{\|\A\|_F^2} \to_{n \to \infty} \frac{k\sigma_{\min}^2(\A)}{\|\A\|_F^2}. \end{equation}
   When $k=1$, this recovers the standard rate $\sigma_{\min}^2(\A)/\|\A\|_F^2$ of Randomized Kaczmarz  \cite{strohmer2009randomized}.
     \end{remark}
We note that even the simple bound \eqref{1-s} is often tighter than the earlier known bound for Gaussian Kaczmarz \cite{rebrova2021block}, especially for larger sketch sizes. Moreover, for matrices $\A$ with known spectral decay, the main bound \eqref{1} gives room for further improving the estimate, via the following known spectral estimate for Randomized SVD error with Gaussian sketching:
\begin{lemma}[\cite{tropp2011structure}, Theorem 10.5]\label{l:spectral1} Let $\A$ be an $m\times n$ matrix with singular values $\sigma_1\geq \sigma_2\geq...$, and let $\S$ be a $k\times m$ Gaussian matrix with $k\geq 2$. For any $2\leq p \leq k-2$, Randomized SVD error $\Err(\A,k)=\E\,\|\A(\I-\P)\|_F^2$ based on $\P=(\S\A)^\dagger\S\A$ satisfies:
\begin{equation*}
  \Err(\A,k) \leq \frac{k-1}{p-1}\cdot\sum_{i\geq k-p}\sigma_i^2.
\end{equation*}
\end{lemma}
For example, in the cases when $\A$ exhibits polynomial or exponential spectral decays, which commonly arise in real-world data matrices, we can get the following corollary (proven in Section~\ref{sec:app-b} of the appendix). To highlight the improved dependence on sketch size $k$, we contrast this with the simple guarantee and let all rates scale with $\sigma_{\min}^2(\A)/\|\A\|_F^2$. 
  \begin{corollary}
  \label{cor:poly-decays} 
 Let $\A$ be a full rank $m\times n$ matrix and let $\S$ be a $k\times m$ Gaussian matrix. 
 There is an absolute constant $C$ such that for any $k$, sketch-and-project \eqref{sp-iteration} satisfies:
 \begin{equation}
 \textnormal{(general spectrum)}\qquad 
 \E\,\|\x_t-\x_*\|^2\leq \bigg(1 -\frac{k\sigma_{\min}^2(\A)}{C\|\A\|_F^2}\bigg)^t\cdot\|\x_0-\x_*\|^2.
 \end{equation}
 If we assume that $\A$ has a polynomial spectral decay of order $\beta> 1$, i.e.,
$\sigma_i^2(\A)\leq c i^{-\beta}\sigma_1^2(\A)$
for all $i$ and some $c>0$, then there is a constant $C=C(\beta,c)$ such that for any $k\leq n/2$:
 \begin{equation}
 \textnormal{(polynomial decay)}\qquad  
 \E\,\|\x_t-\x_*\|^2\leq \bigg(1 -\frac{k^\beta\sigma_{\min}^2(\A)}{C\|\A\|_F^2}\bigg)^t\cdot\|\x_0-\x_*\|^2.
  \end{equation}
 If we assume that $\A$ has an exponential spectral decay of order $\alpha>1$, i.e., $\sigma_i^2(\A)\leq c\alpha^{-i}\sigma_1^2(\A)$ for all $i$ and some $c>0$, then there is a constant $C=C(\alpha,c)$ such that for any $k\leq n/2$:
  \begin{equation}
     \textnormal{(exponential decay)}\qquad
 \E\,\|\x_t-\x_*\|^2\leq \bigg(1 -\frac{\alpha^k\sigma_{\min}^2(\A)}{C\|\A\|_F^2}\bigg)^t\cdot\|\x_0-\x_*\|^2.
\end{equation}
\end{corollary}

We also use Theorem \ref{t:gaussian} to demonstrate that, for certain matrices $\A$, the convergence rate of sketch-and-project can become entirely independent of their condition number (to our knowledge, this is the first result of this kind). This occurs for matrices where fast spectral decay is present only in some part of the spectrum, and from a certain point onward the spectrum flattens out (i.e., a flat-tailed spectrum). This arises for example in data that is distorted by small random noise uniformly in all directions. 
\begin{corollary}\label{c:flat-tailed}
Suppose that  $\A$ has a flat-tailed spectrum, i.e., there is an index $1\leq r\leq n$ such that $\sigma_i^2(\A)\leq c\sigma_{\min}^2(\A)$ for all $i\geq r$ and some $c\geq 1$. Then, there are constants $C=C(c)$ and $C_1$ such that for any $k\geq \max\{2r,C_1\}$, Gaussian sketch-and-project  \eqref{sp-iteration} satisfies:
\begin{equation}
     \textnormal{(flat-tailed spectrum)}\qquad
 \E\,\|\x_t-\x_*\|^2\leq \bigg(1 -\frac{k}{Cn}\bigg)^t\cdot\|\x_0-\x_*\|^2.
\end{equation}
\end{corollary}

 \subsection{Full-spectrum analysis with sub-gaussian sketches}
 \label{s:main-sub-gaussian}
Dense Gaussian sketches present one of the most convenient models for analysis, but they are also the least practical. While it is often a perfect tool for observing the trends that are also exhibited by other sketching distributions, extending the results to more general sketching models is far from trivial. In this section, we present our second main result: a full-spectrum analysis of the expected sketched projection matrix for sketching distributions satisfying certain concentration assumptions. 

  \begin{definition}[Sub-gaussian concentration]\label{d:subgaus}
A random variable $X$ satisfies sub-gaussian concentration with constant $K$ if 
$\inf \left\{t > 0 : \E \exp(X^2/t^2) \leq 2 \right\} \le K.$ A random $n$-dimensional vector $\x$ is $K$-sub-gaussian if $\v^\top\x$ satisfies $K$-sub-gaussian concentration for any unit vector~$\v$.
\end{definition}
 \begin{definition}[Euclidean concentration]\label{d:euclidean}
An $n$-dimensional random vector $\x$ satisfies Euclidean concentration with constant $L$ if for any $m\times n$ matrix $\A$, the random variable
$X=\|\A\x\|-\|\A\|_F$
is $L\|\A\|$-sub-gaussian.
\end{definition}
These assumptions on the rows of the sketching matrix $\S$ are designed so that they naturally capture the behavior of random vectors with i.i.d. sub-gaussian entries (including Gaussian vectors), but also, so that they can be later extended to sparse vectors (more details and properties of such distributions can be found, e.g., in \cite{vershynin2018high}).

 \vspace{2mm}
 
\begin{theorem}\label{t:sub-gaussian}
    Let $\A$ be an $m\times n$ matrix with condition number $\kappa=\sigma_{\max}(\A)/\sigma_{\min}(\A)$ and stable rank $r =  \|\A\|_F^2/\|\A\|^2$, and let $\S$ be a $k\times m$ random matrix with i.i.d.~isotropic $K$-sub-gaussian rows that satisfy Euclidean concentration with constant $L$. Let $\P$ denote the projection onto the span of $\S\A$, and let $\Err(\A, k)$ denote rank $k$ Randomized SVD error \eqref{randsvd_error}.
        There are absolute constants $C_1,C_2>0$ such that if $r\geq C_1\max\{k,L^2\log \kappa,K^4L^2\}$,
    then: 
    \begin{equation*}
    (1-\epsilon)\overbar\P \preceq\E[\P]\preceq (1+\epsilon)\overbar\P,
    \end{equation*}
where $\overbar\P = \gamma_k \A^\top \A (\gamma_k \A^\top \A + \I)^{-1}$,          $\gamma_k = k/\Err(\A,k-1)$, and $\epsilon \leq C_2K^2L/\sqrt r$.
\end{theorem}

As a direct corollary, we obtain the following convergence guarantee for sketch-and-project.
  \begin{corollary}\label{cor:proj-control}
    Under the assumptions of Theorem \ref{t:sub-gaussian},
    sketch-and-project  satisfies:
    \begin{equation}\label{rate-low}
      \Rate(\A,k) \geq (1-\epsilon)
      \frac{\gamma_k\sigma_{\min}^2(\A)}{\gamma_k\sigma_{\min}^2(\A)+1} \geq (1-\epsilon)(1-\frac kn)
      \frac{k\sigma_{\min}^2(\A)}{\Err(\A,k-1)}.
    \end{equation}
      The intermediate expression, which comes directly from the formula for $\overbar\P$, is potentially a sharper estimate of the convergence rate than the estimate derived in Theorem \ref{t:gaussian}. The difference between the two estimates is absorbed by the factor $(1-k/n)$, used to replace the denominator $\gamma_k\sigma_{\min}^2(\A)+1$ with $1$, and it is only substantial when $k$ is close to $n$.
  \end{corollary}
  
As a key motivation, consider matrix $\S_k$ that consists of i.i.d.~mean zero, unit variance and $K$-sub-gaussian entries (sub-gaussian sketch). Then, it satisfies the assumptions of Theorem~\ref{t:sub-gaussian} with constants $K$ and $L=O(K^2)$ (e.g., see Section 4 of \cite{gaussianization}). A practical example of this is the Rademacher sketch, which uses random $\pm1$ entries that can be generated much more efficiently than Gaussian random variables. In this case, both $K$ and $L$ are small absolute constants, and so, the theorem holds for $r\geq O(k+\log\kappa)$ with $\epsilon=O(1/\sqrt r)$. We note that the logarithmic dependence on the condition number $\kappa$ is in general unavoidable due to rare adverse events that occur for discrete sketching distributions.

In addition to the more general sketching distribution, the main difference between Theorem~\ref{t:sub-gaussian} and Gaussian Theorem~\ref{t:gaussian} is that it provides a full-spectrum analysis of the expected projection, i.e., it captures all its eigenvalues rather than just the smallest one. However, to achieve this, we must impose certain regularity assumptions on the data matrix $\A$. Namely, while Theorem~\ref{t:gaussian} only assumed that 
$\A$ has full column rank, here, we require the \emph{stable} rank $r=\|\A\|_F^2/\|\A\|^2$ of $\A$ to be sufficiently large (in particular, larger than the sketch size $k$). Roughly, this prevents a fast spectral decay in the top-$k$ part of the spectrum of $\A$.

\begin{remark}\label{initial-rate} [Initial accelerated convergence of sketch-and-project] While the worst-case rate of convergence of sketch-and-project methods is controlled by the smallest eigenvalue of the expected sketched projection matrix via \eqref{sp-rate}, it is easy to see that all eigenvalues matter for the convergence. Specifically, written in eigenvector basis, the $l$-th component of the expected distance to the solution shrinks with the rate governed by the $l$-th eigenvalue of the expected projection matrix. Indeed, let ${\bf d}_t := {\bf x}_t - {\bf x}_*$. Then the iterate \eqref{sp-iteration} can be equivalently rewritten as
${\bf d}_{t+1} = (\I - (\S\A)^\dagger\S\A){\bf d}_{t}$ (essentially, the iterates ${\bf d}_t$ solve the system ${\bf A} {\bf d} = {\bf 0}$). Then,
\begin{equation*} 
\E\,\langle {\bf d}_{t+1}, {\bf v}_l\rangle 
 = \langle (\I - \E[\P]){\bf d}_{t}, {\bf v}_l\rangle
 = \langle {\bf d}_{t}, {\bf v}_l\rangle - \langle {\bf d}_{t}, \E[\P]{\bf v}_l\rangle
 = (1 - \lambda_l)\langle {\bf d}_{t}, {\bf v}_l\rangle,
\end{equation*} 
where $(\lambda_l, {\bf v}_l)$ is an eigen-pair of $\E[\P]$.
As it was first noted for the Randomized Kaczmarz method in \cite{steinerberger2021randomized}, this implies typically faster convergence than the one guaranteed by \eqref{sp-rate}.
At the same time, the components corresponding to leading singular vectors will shrink faster with the iteration process, eventually bringing the convergence rate to  \eqref{rate-low}. Theorem~\ref{t:sub-gaussian} shows that the initial accelerated convergence rate of sketch-and-project has components of the size
\begin{equation*}
\lambda_l\geq (1-\epsilon)
      \frac{\gamma_k\sigma_{l}^2(\A)}{\gamma_k\sigma_{l}^2(\A)+1}, \quad \text{ for }\quad l = 1, \ldots, n.
\end{equation*}
One can compare this to the initial speed up of Randomized Kaczmarz (where $k=1$), given by $\lambda_l\geq\sigma^2_l(\A)/\|\A\|_F^2$ \cite{steinerberger2021randomized}. Here, again, our result allows us to quantify the dependence of initial accelerated convergence on the sketch size $k$ thanks to the term $\gamma_k=k/\Err(\A,k-1)\geq k/\|\A\|_F^2$.
\end{remark}

The proof of Theorem~\ref{t:sub-gaussian} is inspired by techniques from asymptotic random matrix theory used to analyze Stieltjes transforms \cite{silverstein1995empirical}. This approach yields $\overbar\P$, which we call the \emph{surrogate expression} for the expected projection matrix, since it can be viewed as an asymptotically exact formula if we let $m,n,r\rightarrow\infty$. As part of this analysis, we must control the concentration of random quadratic forms such as $\s_i^\top\B\s_i$, where $\s_i^\top$ is a row of the sketching matrix $\S$ and $\B$ is some positive semidefinite matrix. For i.i.d. sub-gaussian sketches, such concentration can be ensured, e.g., by the classical Hanson-Wright inequality \cite{rudelson2013hanson}. In our case, this is ensured by the more general condition of Euclidean concentration imposed on the rows of $\S$.

\subsection{Extension to sparse sketches}
\label{s:main-sparse}
In our next result, we show that our full-spectrum analysis of expected sketched projection applies not only to dense sub-gaussian sketches, but also to certain sparse sketching matrices.  We rely on a sparse sketching technique called Leverage Score Sparsified (LESS) embeddings, which was recently introduced by \cite{less-embeddings} (the below definition is based on \cite{gaussianization}). These sketches can be implemented much more efficiently, reducing the per-iteration cost of sketch-and-project, and can be viewed as a bridge between the fast but less stable block sketches and heavy but more stable Gaussian sketches.
\begin{definition}[LESS embeddings \cite{less-embeddings,gaussianization}]\label{d:leverage-score-sparsifier}
Consider a full rank $m\times n$ matrix $\A$ and let
$l_i=\a_i^\top(\A^\top\A)^{-1}\a_i$ for $i=1,...,m$  denote its
leverage scores. Fixing absolute constants $C,K\geq 1$ and letting $\e_j\in\R^m$ be the $j$-th standard basis vector, we define a LESS embedding for $\A$  with sketch size $k$ and $s$ non-zeros per row as a matrix $\S$ with $k$ i.i.d.\ row vectors:
\begin{equation*}\s_i^\top=\frac1{\sqrt{k}}\sum_{j=1}^s\frac{r_{i,j}}{\sqrt{sp_{t_{i,j}}}}\e_{t_{i,j}}^\top\quad\text{for}\quad i=1,...,k,\end{equation*}
where $r_{i,j}$ are i.i.d.\ $K$-sub-gaussian random variables with mean zero and unit variance, and $t_{i,j}$ are sampled i.i.d.~from a probability distribution $p=(p_1,...,p_m)$ such that
$p_i\geq l_i/(Cn)$.
\end{definition}

We show that, when $m\gg n$, then a sparse LESS embedding sketch retains the same spectral properties of the expected sketched projection matrix as a dense sub-gaussian sketch. The proof of the following Theorem~\ref{c:less} is in Section~\ref{s:less-proof}.
    \begin{theorem}\label{c:less}
  Let $\A$ be an $m\times n$ matrix with condition number $\kappa=\sigma_{\max}(\A)/\sigma_{\min}(\A)$ and stable rank $r =  \|\A\|_F^2/\|\A\|^2$. Let $\P$ denote the projection onto the span of $\S\A$, and let $\Err(\A, k)$ denote rank $k$ Randomized SVD error, as per \eqref{randsvd_error}. There are absolute constants $C_1,C_2,C_3 > 0$ such that if $\S$ is a LESS embedding of size $k$ with $s\geq C_1 n\log(\kappa n)$ non-zeros per row, and if $r \ge C_2\max\{k, \log \kappa\}$, then
      \begin{align*}
    (1-\epsilon)\overbar\P \preceq\E[\P]&\preceq (1+\epsilon)\overbar\P,
    \\
    \text{where} \qquad \overbar\P = \gamma_k \A^\top \A (\gamma_k \A^\top \A + \I)^{-1},
         \qquad\gamma_k &= \frac{k}{\Err(\A,k-1)},\qquad\text{and} \qquad\epsilon \leq \frac{C_3}{\sqrt r}.
    \end{align*}
\end{theorem}

\begin{remark}[Generating LESS sketches]\label{r:less}
For a LESS embedding matrix $\S$ of size $k$ with $s$ non-zero entries per row, computing the matrix product $\S\A$ costs $O(kns)$ operations. For $s=1$, this is equivalent to computing a block sketch (that samples $k$ rows of the matrix $\A$), whereas for $s=m$, it effectively corresponds to a dense sub-gaussian sketch, so LESS embeddings can be viewed as interpolating between those two extremes. To obtain a probability distribution $p=(p_1,...,p_m)$ satisfying the condition from Definition \ref{d:leverage-score-sparsifier}, we can use one of two preprocessing steps:
\end{remark} \vspace{-2mm}
\begin{enumerate}
    \item \emph{Approximating leverage scores.}
    This can be done using standard techniques from RandNLA \cite{DMMW12_JMLR,cw-sparse} in time $O(\nnz(\A)\log m + n^3\log n)$, where $\nnz(\A)$ is the number of non-zeros in $\A$. This is less than the cost of solving the system when $m\gg n$. Then, we can define $p$ using the obtained leverage score estimates.
    \item \emph{Randomized preconditioning.} Replacing $\A$ and $\b$ with $\tilde\A=\H\D\A$ and $\tilde\b=\H\D\b$, where $\H$ is an $m\times m$ fast randomized Fourier/Hadamard transform and $\D$ is diagonal with random $\pm1$ entries, does not change the solution $\x_*$, and with high probability, ensures that all the leverage scores of $\tilde \A$ are nearly uniform \cite[Lemma 3.3]{tropp2011improved}. This operation costs $O(mn\log m)$ time and allows us to use a uniform distribution $p$. Similar preconditioning was proposed for Block Kaczmarz methods \cite{needell2014paved}.  Note that in practice (and in our experimental section), it is typical to omit preconditioning and rely on the incoherence of many natural datasets.
\end{enumerate}

We would also like to note that the dichotomy between block sketching (light and more dependent on the matrix structure) and dense sketching (slow and heavy, but more general and theoretically better understood) is notorious across the numerical methods involving sketching, e.g., \cite{rebrova2019sketching, haddock2021greed}. 
The results we show here for LESS sketches are as succinct as for the dense case and theoretically allow for the sparsity $s$ that is almost of the order of $n$ non-zeros per $m$-dimensional row (where $m\gg n$ for a highly overdetermined linear system). Nevertheless, empirical results in Section \ref{s:experiments} suggest that this sparsity level is still quite conservative, and our spectral analysis in practice applies to sketches where $s$ is a small constant, which effectively replicates a block sketching procedure. It is also natural to ask whether our results can be extended to other popular sparse sketching operators, such as the CountSketch \cite{cw-sparse}, or to other structured sketching operators such as subsampled randomized Hadamard/Fourier transforms \cite{ailon2009fast}. We leave this as a direction for future work. 

    \subsection{Extension to stochastic Newton methods}
\label{s:main-newton}
 Our convergence analysis for sketch-and-project can be easily adapted to a number of randomized convex optimization algorithms from the  literature \cite{Gower2019,hanzely2020stochastic,gower2021adaptive,sketched-newton-raphson}. In this setting, the goal is to find $\x_*=\argmin_\x f(\x)$, for some convex function $f:\R^m\rightarrow \R$. Naturally, such a problem can be solved by repeatedly solving the Newton system, i.e., finding $\x_{t+1}$ such that $\nabla^2f(\x_t)(\x_{t+1}-\x_t)=\nabla f(\x_t)$ for $t=0,1,2$, etc. To avoid the cost of solving the system exactly at each iteration, we can sketch it with a $k\times m$ random matrix $\S$. For example, \cite{Gower2019} proposed the following so-called Randomized Subspace Newton:
\begin{equation*}
    \x_{t+1} = \argmin_{\x\in\R^m}\|\x-\x_t\|_{\nabla^2 f(\x_t)}^2
    \qquad\text{such that}\qquad
    \S\,\nabla^2f(\x_t)(\x_t-\x) = \eta_t\cdot\S\,\nabla f(\x_t),
\end{equation*}
where $\eta_t$ is an appropriately chosen step size. The expected convergence rate of this procedure is controlled by a quantity analogous to $\lambda_{\min}(\E[\P])$, except with the input matrix $\At$ replaced by the positive semidefinite square root of the Hessian matrix $\H_t = \nabla^2 f(\x_t)$:
\begin{equation*}
    \E\big[f(\x_{t+1})-f(\x_*)\big] \leq \big(1- c\rho_t)\cdot \E\big[f(\x_t)-f(\x_*)\big]\quad \text{for}\;\rho_t = \lambda_{\min}^+\big(\H_t^{1/2}\E[\S^\top(\S\H_t\S^\top)^\dagger\S]\H_t^{1/2}\big),
\end{equation*}
where $\lambda_{\min}^+$ denotes the smallest positive eigenvalue, whereas $c\in(0,1)$ is a constant that depends on the smoothness and strong convexity of $f$. Also, \cite{hanzely2020stochastic} showed that, for a similar procedure with added cubic regularization, the convergence improves as we approach $\x_*$, so that $c\rightarrow 1$ and we can replace the changing Hessians $\H_t$ in the definition of $\rho_t$ with the fixed Hessian at the optimum, $\nabla^2 f(\x_*)$. Our guarantees from Theorems~\ref{t:gaussian},~\ref{t:sub-gaussian}, and \ref{c:less} can be used to complete the above convergence characterization when $\S$ is Gaussian, sub-gaussian or LESS, by lower bounding the $\rho_t$ quantity, as shown below.

\begin{corollary}\label{c:newton}
    Let $\H$ be an $m\times m$ positive semidefinite matrix with rank $n$, and let $\S$ be a $k\times m$ sketching matrix $\S$. Moreover, suppose that one of the following is true: 
\begin{enumerate}
        \item $\S$ is a Gaussian matrix and $\epsilon = \frac{4k}{n}+\frac{8\log(3n)}{n}$;
        \item $\S$ is a sub-gaussian or LESS embedding matrix and $\epsilon=O(1/\sqrt r + k/n)$, where $r=\tr(\H)/\|\H\|$ satisfies         
        $r\geq C\max\{k,\log\kappa\}$, with $\kappa=\lambda_{\max}(\H)/\lambda_{\min}^+(\H)$.
    \end{enumerate}
    Then, letting $\Err(\A, k)$ be the rank-$k$ Randomized SVD error, as per \eqref{randsvd_error}, we have:
    \begin{equation*}
        \lambda_{\min}^+\big(\H^{1/2}\E[\S^\top(\S\H\S^\top)^\dagger\S]\H^{1/2}\big)\geq (1-\epsilon)\frac{k\,\lambda_{\min}^+(\H)}{\Err(\H^{1/2},k-1)}
        \geq\ \frac{k\,\lambda_{\min}^+(\H)}{\tr(\H)}.
    \end{equation*} 
    \end{corollary}   
\noindent
The proof of Corollary~\ref{c:newton} is almost immediate. It follows from Theorems \ref{t:gaussian}, \ref{t:sub-gaussian} and \ref{c:less} by setting $\A=\bar\U\bar\D^{1/2}$, where $\K=\U\D\U^\top$ is the eigendecomposition of $\H$, whereas $\bar\U\in \R^{m\times n}$ and $\bar\D\in\R^{n\times n}$ are truncated versions of $\U$ and $\D$ produced by omitting the parts corresponding to the $0$ eigenvalue (if $\K$ is not full-rank).

\section{Proofs of theoretical results}\label{s:all-proofs}

Here, we give proofs for the results stated in Section~\ref{s:main-results}.

  \subsection{Gaussian sketching: proof of Theorem~\ref{t:gaussian}}\label{s:proofs-gaus}
We need two auxiliary results whose proofs are deferred to Section ~\ref{sec:app_a} in the appendix. The first one uses rotational invariance of Gaussian distribution to let us pick a convenient basis for the sketched projection matrix. It follows by an argument similar to \cite[Section 3.1]{cook2011}.
 \begin{lemma}\label{l:tensor-function}
Let $\A$ be an $m\times n$ rank-$n$ matrix, and $\A = \U \D\V^\top$ its compact SVD. Let $\S$ and $\Z$ be $k\times m$ and $k\times n$ Gaussian matrices, with $\P = (\S\A)^\dagger\S\A$ and $\tilde\P=(\Z\D)^\dagger\Z\D$. Then:
\begin{enumerate}
    \item The matrices $\E[\P]$ and $\E[\tilde\P]$ share the same spectrum.
    \item The matrices $\E[\tilde\P]$ and $\D$ share the same (standard) eigenbasis.
    \item The Gaussian Randomized SVD error \eqref{randsvd_error} satisfies $\Err(\A, k) = \Err(\D, k)$.
\end{enumerate}
\end{lemma}
\noindent
The second auxiliary result is our key proposition that connects a form of the expected sketched matrix with the standard Randomized SVD error (compare with \eqref{randsvd_error}):
\begin{proposition}\label{lem:schur-1}
 Let $\A$ be $m\times n$, $\S$ be a $k\times m$ Gaussian matrix, and $\P = (\S\A)^\dagger\S\A.$ Then, for any event $\mathcal{E}$ such that the rows of $\S\A$ are identically distributed conditioned on $\mathcal{E}$, we have
\begin{equation*}\E\left[\tr\,(\S\A\A^\top\S^\top)^{-1}\mid\mathcal{E}\right]\ge  \frac{k\, \mathbb{P}( \EE)}{\Err(\A, k-1)}.\end{equation*}
\end{proposition}
\begin{proof}[Proof of Proposition~\ref{lem:schur-1}]
 Let $\x_i \in \R^n$ be the $i$-th row of a random matrix $\X := \S\A \in \R^{k \times n}$. By a standard Schur complement result (e.g., \cite{boyd2004convex}) applied to a block decomposition of the matrix $(\S\A\A^\top\S^\top)^{-1} = (\X\X^\top)^{-1} \in \R^{k \times k}$ separating an $(i,i)$-th entry into a $1 \times 1$ block, we have for any $i = 1, \ldots, k$, 
\begin{equation*}
[(\X\X^\top)^{-1}]_{ii} = (\|\x_i\|^2 - \x_i^\top \X_{-i}^\top(\X_{-i}\X_{-i}^\top)^{-1}\X_{-i}\x_i)^{-1}  \!= (\|\x_i\|^2 - \x_i^\top \X_{-i}^\dagger  \X_{-i}\x_i)^{-1} \!= \frac{1}{\x_i^\top \P_{\perp i} \x_i},
\end{equation*}
where $\X_{-i} \in \R^{(k-1) \times n}$ is the matrix obtained by deleting the $i$-th row from $\X$, $\P_{-i} = (\X_{-i})^\dagger \X_{-i}$, and $\Pperp{i}:=\I-\P_{-i}$.
Since for all $i$ both $\x_i$ and $\P_{\perp i}$ have the same distribution conditioned on the event $\mathcal{E}$, we can conclude that 
 \begin{equation*}\E[\tr(\X\X^\top)^{-1}\mid\mathcal{E}] = \sum_{i=1}^k\E[(\X\X^\top)^{-1}\mid\mathcal{E}]_{ii} = \E\left[\frac{k}{\x_k^\top \P_{\perp k} \x_k}\,\Big|\,\mathcal{E}\right].\end{equation*}
 Now, using that $x \mapsto 1/x$ is convex over $\R_+$ together with Jensen's inequality,
 \begin{equation*}
 \E\left[\frac{k}{\x_k^\top \P_{\perp k} \x_k}\,\Big|\, \mathcal{E}\right] \ge  \frac{k}{\E[\x_k^\top\P_{\perp k}\x_k\mid \EE]} \geq \frac{k \mathbb{P}( \EE)}{\E[\x_k^\top\P_{\perp k}\x_k]}
  \end{equation*}
  and the claim is concluded as follows:
\begin{align*}
\E[\x_k^\top\P_{\perp k}\x_k] &= \E\big[\E_{\x_k}( \x_k^\top\P_{\perp k}\x_k\mid\Pperp{k})\big]= \E\big[\tr (\A\Pperp{k}\A^\top)\big] \\
&= \E\big[\tr (\A\Pperp{k}\Pperp{k}^\top\A^\top)\big] = \E\|\A\Pperp{k}\|_F^2.
\end{align*}
Since $\A\Pperp{k} = \A(\I - \P^{(k-1)})$ for a $k-1 \times m$ standard Gaussian matrix $\P^{(k-1)}$, this concludes the proof of Proposition~\ref{lem:schur-1}. 
\end{proof}
\begin{proof}[Proof of Theorem~\ref{t:gaussian}]
First, rotational invariance of the standard Gaussian matrix $\S$ allows us to pass to a very convenient basis, as shown in Lemma~\ref{l:tensor-function}. Therefore, with a slight abuse of notation, we will assume that $\S$ is a $k\times n$ Gaussian matrix, $\A=\D$ is an $n\times n$ diagonal matrix and $\sigma_1\geq \sigma_2\geq ...\geq \sigma_n > 0$ are its diagonal entries (and eigenvalues), $\P = (\S\D)^{\dagger}(\S\D)$. Let $\{\e_1,...,\e_n\}$ be the standard Euclidean eigenbasis of both $\E[\P]$ and $\D$. By Lemma~\ref{l:tensor-function}, we will conclude the statement of the theorem by showing that 
\begin{equation}\label{1}
  \lambda_{\min}(\E[\P]) = \min_i\E[\e_i^\top\P\e_i]\geq (1-\epsilon) \frac{k\lambda^2_{\min}(\D)}{\Err(\D,k-1)}.
\end{equation}

\vspace{0.15cm}
\noindent\textbf{Step 1. Sherman-Morrison lemma and Gaussian quadratic forms.}
By the definition of pseudoinverse,  
  \begin{equation*}
    \e_i^\top\P\e_i = \e_i^\top\D\S^\top(\S\D^2\S^\top)^{-1}\S\D\e_i
=\sigma_i^2\hat\s_i^\top(\S\D^2\S^\top)^{-1}\hat\s_i = \sigma_i^2\hat\s_i^\top\Y\hat\s_i,
  \end{equation*}
 where $\hat\s_i$ denotes the $i$-th column vector of $\S$, and $\Y := (\S\D^2\S^\top)^{-1}$. Then, by a version of the Sherman-Morrison rank-one update formula (Lemma~\ref{lem:rank-one} in the appendix), 
\begin{equation*}
   \e_i^\top\P\e_i =\sigma_i^2\hat\s_i^\top(\S\D^2_{-i}\S^\top+\sigma_i^2\hat\s_i\hat\s_i^\top)^{-1}\hat\s_i= \frac{\sigma_i^2\hat\s_i^\top(\S\D^2_{-i}\S^\top)^{-1}\hat\s_i}
    {1+\sigma_i^2\hat\s_i^\top(\S\D^2_{-i}\S^\top)^{-1}\hat\s_i},
\end{equation*}
  where $\D^2_{-i}$ denote $\D^2$ with $i$th diagonal entry zeroed out. Then, since $\S\D^2_{-i}\S^\top$ is independent from $\hat\s_i$, and $\S\D^2_{-i}\S^\top\preceq \S\D^2\S^\top= \Y^{-1}$, we get:
  \begin{equation}\label{step1}
      \min_i\E[\e_i^\top\P\e_i]
   \geq \min_i\ 
    \E\bigg[\frac{\sigma_i^2\s^\top\Y\s}
    {1+\sigma_i^2\s^\top\Y\s}\bigg]
    =\E\bigg[\frac{\sigma_n^2\s^\top\Y\s}
    {1+\sigma_n^2\s^\top\Y\s}\bigg],
  \end{equation}
 where $\s$ is a new $k$-dimensional Gaussian vector independent of $\S$ (and so, of $\Y$), and we used that a function $f(\sigma) = \sigma^2 x/(1 + \sigma^2 x)$ is increasing for all $x > 0$.
  
  \vspace*{0.2cm}
  
  \noindent\textbf{Step 2. Gaussian Hanson-Wright inequality and removing dependence on $\s$.} Note that $\E\s^\top\Y\s = \tr(\Y)$. By a Gaussian version of the Hanson-Wright inequality (Lemma~\ref{lem:hw_tail} in the appendix), for any $\alpha > 0$, the event 
\begin{equation}\label{gaus_hw}\EE_{\alpha} := \{\s^\top\Y\s \leq w_\alpha(\Y)\} \quad \text{ with }\quad w_\alpha(\Y) :=\tr\,\Y + \sqrt{2\alpha\,\tr\,(\Y^2)}+\alpha\,\|\Y\|
\end{equation}
  has $\P(\EE_{\alpha}|\Y) \ge 1 - e^{-\alpha/2}$ and $\E[\s^\top\Y\s\cdot\one_{\neg\EE_\alpha}|\Y] \le 5 \tr\,(\Y)\cdot e^{-\alpha/2}$. This implies
  \begin{align*}
    \E_\s\Big[ \frac{\sigma_n^2 \s^\top\Y\s}
    {1+\sigma_n^2 \s^\top\Y\s}\mid \Y\Big]
    &\geq  \E_\s\Big[ \frac{\s^\top\Y\s}
    {\sigma_n^{-2}+\s^\top\Y\s} \one_{\EE_\alpha}\mid \Y\Big]\\
    &
    \ge\frac{\tr\,\Y - \E[\s^\top\Y\s \cdot \one_{\neg {\EE_\alpha}}\mid \Y]}
      {\sigma_n^{-2} + w_\alpha(\Y)} 
      \ge (1 - 5 e^{-\alpha/2})\frac{\tr\,\Y }
      {\sigma_n^{-2}+w_\alpha(\Y)}.
  \end{align*}
  Further, let $\mathcal{U}_\alpha:= \{w_\alpha(\Y) \le t_\alpha\}$ for some $t_\alpha$ to be defined later. Note that:

\begin{equation*} 
 \E_\Y\Big[\,\frac{\tr\,\Y}{\sigma_n^{-2}+w_\alpha(\Y)}\Big]  \geq\PP(\mathcal{U}_\alpha)
       \E\Big[\,\frac{\tr\,\Y}{\sigma_n^{-2}+w_\alpha(\Y)}\mid\mathcal{U}_\alpha\Big]
        \geq\PP(\mathcal{U}_\alpha)
       \frac{\E[\tr\,\Y\mid\mathcal{U}_\alpha]}{\sigma_n^{-2}+t_\alpha}.
\end{equation*}
  So, by \eqref{step1} and independence of $\Y$ and $\s$,
    \begin{equation}\label{step2}
      \min_i\E[\e_i^\top\P\e_i] =  \min_i\E_\Y(\E_\s[\e_i^\top\P\e_i| \Y])\ge 
       (1-5e^{-\alpha/2})\PP(\mathcal{U}_\alpha)
       \frac{\E[\tr\,\Y\mid\mathcal{U}_\alpha]}{\sigma_n^{-2}+t_\alpha}.
      \end{equation}
  
  \noindent\textbf{Step 3. Using Proposition \ref{lem:schur-1} and final probability estimates.}  Note that both trace and spectral norm of $\Y$ are invariant with respect to the permutations of the rows of $\S\Sigmab^{1/2}$, so conditioning on the event $\mathcal{U}_\alpha$ keeps the rows identically distributed and we can apply Proposition~\ref{lem:schur-1} to estimate the leftover expectation term. As a result, we have that  
  \begin{equation*}
  \E\big[\,\tr\,\Y\mid \mathcal{U}_\alpha\big]\ge\frac{k\,\PP(\mathcal{U}_\alpha)}{\Err(\D, k-1)}
  \end{equation*} 
  and
    \begin{equation}\label{3}
   \min_i\E[\e_i^\top\P\e_i] \ge  \frac{(1-5e^{-\alpha/2})\mathbb{P}^2(\mathcal{U}_\alpha)}{1+t_\alpha\sigma_n^2} \cdot \frac{k\lambda^2_{\min}(\D)}{\Err(\D, k-1)}.
    \end{equation}
  To conclude the desired lower bound \eqref{1}, we only need to choose $t_\alpha$ to make $\PP(\mathcal{U}_\alpha) = \PP(\{\tr\, \Y+ \sqrt{2\alpha\,\tr\,\Y^2}+\alpha\,\|\Y\| \le t_\alpha\})$ sufficiently large so that 
\begin{equation}\label{4}
\frac{(1-5e^{-\alpha/2})\mathbb{P}^2(\mathcal{U}_\alpha)}{1+t_\alpha\sigma_n^2} \ge 1- \epsilon.
\end{equation}
Direct computation involving a concentration inequality for Gaussian matrices (Lemma~\ref{lem:gaus-compute} in the appendix) shows that choosing 
\begin{equation*}
t_\alpha := \frac{\sigma_n^{-2}\cdot(k+\sqrt{2\alpha k} + \alpha)}{(\sqrt{n} - \sqrt{k} - \sqrt{\alpha})^2} \quad \text{ we have } \quad \epsilon \le  \frac{4k}n + \frac{8\log(3n)}n
\end{equation*}
and thus Theorem~\ref{t:gaussian} follows from $\eqref{3}$ and $\eqref{1}$.
\end{proof}

\subsection{Sub-gaussian sketching: proof of Theorem \ref{t:sub-gaussian}}
\label{s:subgaus-proof}
Let us first set up some notations. Recall that $\P = (\S\A)^\dagger\S\A$ and define its surrogate  $\overbar\P:=\gamma\Sigmab(\gamma\Sigmab+\I)^{-1}$, where $\Sigmab=\A^\top\A$ and $\gamma=\gamma_k:=k/\E\,\|\A\Pperp{k}\|_F^2$. Let $\sigma_1^2\geq \sigma_2^2\geq ...\geq \sigma_n^2$ be the eigenvalues of $\Sigmab$.

Let $\x_i \in \R^n$ be the $i$-th row of $\S\A$, i.e., $\x_i = \A^\top \s_i$ for $\s_i \in \R^m$ the $i$-th row of $\S$. Let us denote a ``truncated" projection with one row less by  $\P_{-i} = (\S_{-i}\A)^\dagger \S_{-i}\A$ for $\S_{-i} \in \R^{(k-1) \times n}$ the matrix obtained by deleting the $i$-th row from $\S$. We will also frequently refer to the residual projections and denote them by $\P_\perp = \I-\P$, $\overbar\P_\perp:=\I-\overbar\P$, and $\Pperp{i}:=\I-\P_{-i}$. The next lemma is proved in the appendixary material: 

\begin{lemma}\label{l:def-delta}
Let $\A$ be an $m\times n$ matrix with $\Sigmab=\A^\top\A$. Let $\S$ be a $k\times m$ random matrix with i.i.d.~rows that satisfy $L$-Euclidean concentration (Definition~\ref{d:euclidean}). Consider the events
\begin{equation}\label{ei}
    \mathcal{E}_i:= \quad\Big\{\x_i^\top\Pperp{i}\x_i\geq \frac12 \tr\Sigmab\Pperp{i}\Big\},  \quad\text{ and }\quad \mathcal{E} :=\bigwedge_{i=1}^k\mathcal{E}_i. 
\end{equation}
Then, defining $r =  \|\A\|_F^2/\|\A\|^2$ as the stable rank of $\A$, we have
\begin{align}
    \text{(almost surely)}\quad\tr\Sigmab\Pperp{i} &\ge \|\Sigmab\|^2  (r-k) \label{trace-key}
    \\[2mm]
    \text{and}\quad\delta := \PP(\neg \EE) &\le 2\exp(-cr/L^2).\label{delta-def}
\end{align}
\end{lemma}

\begin{proof}[Proof of Theorem~\ref{t:sub-gaussian}]
Without loss of generality (rescaling both $\A$ and $\gamma$), we can assume:
\begin{equation}\label{norm}
\|\Sigmab\| = 1 \quad \Rightarrow \quad \|\Sigmab^{-1/2}\|= \kappa; \quad \tr(\Sigmab) = r.
\end{equation}

\textbf{1. Symmetrization, initial conditioning and the three part splitting.} 
The claim of Theorem~\ref{t:sub-gaussian} is equivalent to proving that, for some universal constant $C > 0$,
\begin{equation}\label{t2-sym}
  \epsilon := \|\overbar\P^{-1/2}\E[\P]\overbar\P^{-1/2}-\I\| \le CLK^2/\sqrt r.
\end{equation}

Now, let us pass to a version of \eqref{t2-sym} conditioned on the event $\Ec$ defined as in \eqref{ei}. Here and further, $\E_\EE$ denotes corresponding conditional expectation. By a technical Lemma~\ref{lem:cond-symm} proved in the appendix using the law of total probability, 
\begin{equation*}
\|\overbar\P^{-1/2}\E[\P]\overbar\P^{-1/2}-\I\|\leq \|\overbar\P^{-1/2}\E_\EE[\P]\overbar\P^{-1/2}-\I\|
   + \|\overbar\P^{-1}\|\cdot 2\delta
\end{equation*}
where $\delta = \PP(\neg \EE) \le 2\exp(-cr/L^2)$ by \eqref{delta-def} and $\|\overbar\P^{-1}\|\leq 1+r\kappa^2$.  So,
\begin{equation}\label{rho-rhoeps}
\epsilon \le \epsilon_\Ec + 2(1 + r \kappa^2)\delta, \quad \text{ where } \quad \epsilon_\Ec := \|\overbar\P^{-1/2}\E_\EE[\P]\overbar\P^{-1/2}-\I\|.\end{equation}
By the conditions of Theorem~\ref{t:sub-gaussian} we have $r\geq CL^2\log\kappa$ for sufficiently large constant $C>0$, so $\delta r\kappa^2\leq \exp(2\log\kappa+\log r-cr/L^2)\leq\exp(-c'r/L^2).$ So, to get \eqref{t2-sym}, it is enough to prove:
\begin{equation*}\epsilon_\Ec \le CLK^2/\sqrt r.
\end{equation*}

The bound on $\epsilon_\Ec$ relies on the following three-part split, proved in Lemma~\ref{lem:three-part-splitting}, using a Sherman-Morrison type rank-one update formula for the Moore-Penrose pseudoinverse:
\begin{align*}
\epsilon_\Ec \le T_1 + T_2 + T_3, \quad 
\text{where}\quad T_1 &:= \left\|\Sigmab^{-1/2}\E_{\Ec}\big[ \Pperp{k} \x_k\x_k^\top-\Pperp{k}\Sigmab\big]\Sigmab^{-1/2}\right\|, \\
         \quad \quad \quad T_2 &:= \left\|\Sigmab^{-1/2}\E_{\Ec}[\P_{\perp k}\Sigmab - \P_\perp\Sigmab]\Sigmab^{-1/2}\right\|, \\
        T_3 &:= \left\|\E_{\Ec}\bigg[\frac{\E[\xi] - \xi}{\xi} \cdot
        \Sigmab^{-1/2}\Pperp{k} \x_k\x_k^\top \Sigmab^{-1/2}\bigg]\right\| ,
\end{align*}
with $\xi := \x_k^\top\Pperp{k}\x_k$ and $\E[\xi] = \E\tr(\Sigmab \Pperp{k}) = \E\|\A\Pperp{k}\|_F^2 =k/\gamma$ by definition of $\gamma$.

  \vspace*{0.2cm}
  
    \textbf{2. Bounding $T_1$.} First, note that the conditioning is the only reason for $T_1$ to be non-zero. Indeed, since the rows of $\S$ are isotropic and $\x_k = \A^\top\s_k$, we have:
    \begin{equation*}
    \E[\Pperp{k}\E[\x_k\x_k^\top-\Sigmab\mid\Pperp{k}]]=\zero.\end{equation*}
    Carefully isolating the quadratic form $\x_k\x_k^\top = \s_k \A^\top \A \s_k$, its tail is bounded by the Euclidean concentration assumption to have $T_1 \le \kappa^2(\delta + CL^2 r \delta)/(1 - \delta)$.

 Indeed, let $\Xi:=\Pperp{k}(\x_k\x_k^\top-\Sigmab)$, by the law of total probability, 
    \begin{equation*}
   \E_\EE[\Xi] = - \E[\one_{\neg\Ec}\Xi]/ \PP(\Ec) \le \frac{\E[\one_{\neg\Ec}\Xi]}{1 - \delta}.
    \end{equation*}
Using this estimate, along with  $\|\Pperp{k}\|=1$ as a projection matrix and $\|\Sigmab\| = 1$ \eqref{norm},  we have
\begin{align*}
 T_1 &\le
\kappa^2\, \|\E_{\Ec} [\Pperp{k}(\x_k\x_k^\top-\Sigmab)]\|
\leq \frac{\kappa^2}{1-\delta}\|\E[\one_{\neg\Ec}\Pperp{k}(\x_k\x_k^\top-\Sigmab)])\|
\\ &\leq  \frac{\kappa^2}{1-\delta}
\E[\one_{\neg\Ec}\|\Pperp{k}\|(\|\x_k\|^2+\|\Sigmab\|)] \leq\frac{\kappa^2}{1-\delta}(\delta + \E[\one_{\neg\Ec}\|\x_k\|^2]).
\end{align*}
 To analyze the final expectation, we use Euclidean concentration: \begin{equation*}\PP(\|\x_k\|^2>t) = \PP(\|\s_k^\top\A \A^\top \s_k\|^2 >t ) \le 2\exp(-ct/L^2)\end{equation*} for any $t\geq 2 \|\A^\top\|_F^2 = 2\|\A\|_F^2 = 2\tr(\Sigmab)=2r$, so:
\begin{equation*}
\E[\one_{\neg\Ec}\|\x_k\|^2] 
=\int_0^\infty\PP(\one_{\neg\Ec}\|\x_k\|^2>t)dt 
     \leq 2r\delta + \int_{2r}^\infty \PP(\|\x_k\|^2>t)dt 
     \leq 2r\delta + \int_{2r}^\infty 2\exp(-ct/L^2)dt,
\end{equation*}
that is upper bounded by $CL^2r\delta$ with some $C>0$. So, $T_1 \le \kappa^2(\delta + CL^2 r \delta)/(1 - \delta)$.

  \vspace*{0.2cm}
  
 \textbf{3. Bounding $T_2$.} By the rank-one update formula given in Lemma~\ref{lem:rank-one-update}, we have 
\begin{equation}
	\E_{\Ec}[\P_{\perp k} - \P_\perp] = \E_{\Ec}[\P - \P_{-k}]
	= \E_{\Ec}\bigg[ \frac{\Pperp{k} \x_k \x_k^\top \Pperp{k} }{\x_k^\top \Pperp{k} \x_k} \bigg]\preceq \frac{4}{r}\cdot \E_{\Ec}[\Pperp{k}\Sigmab\Pperp{k}],\label{sigmas}
	\end{equation}
	since, given $\Ec$, $\x_k^\top \Pperp{k} \x_k \geq\frac12\tr\Sigmab\Pperp{k}\geq \frac12(r-k)\geq r/4$.
  Next, we use two technical observations proved in Lemma~\ref{lem:subgaus-sp-ext}, which hold almost surely, using that $r\geq 2k$ and $\|\Sigmab\|=1$: 
\begin{equation}\label{third_one}
\Pperp{k}\Sigmab\Pperp{k}\preceq 2(\Sigmab+\P)\qquad\text{and}\qquad \Sigmab^{-1/2}\preceq \sqrt{\gamma}\,\overbar\P^{-1/2}\preceq\overbar\P^{-1/2}.
\end{equation}
Also, we have
  $\E_{\Ec}[\Pperp{k}-\P_\perp]
    \Sigmab\E_{\Ec}[\Pperp{k}-\P_\perp]\preceq\E_{\Ec}[\Pperp{k}-\P_\perp]$. Putting this together:
\begin{align*}
    T_2  
    &= \sqrt{\|\Sigmab^{-1/2}\E_{\Ec}[\Pperp{k}-\P_\perp]
    \Sigmab\E_{\Ec}[\Pperp{k}-\P_\perp]\Sigmab^{-1/2}\|}
\leq \sqrt{\|\Sigmab^{-1/2}\E_{\Ec}[\Pperp{k}-\P_\perp]
    \Sigmab^{-1/2}\|}
    \\
    &\leq \frac 4{\sqrt r}\cdot \sqrt{\|\Sigmab^{-1/2}\E_{\Ec}[\Sigmab+\P]\Sigmab^{-1/2}\|}
 \leq \frac{C}{\sqrt r}\cdot \sqrt{1+\|\overbar\P^{-1/2}\E[\P]\overbar\P^{-1/2}\|} =: \frac{C}{\sqrt r} \cdot\sqrt{1 + \epsilon_1}.
\end{align*}
We will return to bounding $\epsilon_1$ at the end of the proof, as it requires a recursive argument.

\vspace{0.2cm}
  
 \textbf{4. Bounding $T_3$.}
By Cauchy-Schwartz, letting $S^{n-1}$ denote the unit sphere in $\R^n$:
\begin{align}
T_3 &=
\left\|\E_{\Ec}\left(\frac{\E\xi - \xi}{\xi} \cdot
        \Sigmab^{-1/2}\Pperp{k} \x_k\x_k^\top \Sigmab^{-1/2}\right)\right\|  \nonumber\\
    &\leq \sup_{\u,\v\in S^{n-1}} \E_\mathcal{E}\left[\frac{\left|\E\xi - \xi\right|}{\xi}\cdot
    \left|\v^\top\Sigmab^{-1/2}\Pperp{k}\x_k\x_k^\top\Sigmab^{-1/2}\u\right|\right]
    \nonumber\\
        &\leq  \sqrt{\E_\mathcal{E}
        \left[\frac{|\E\xi - \xi|^2}{\xi^2}\right]}\cdot\sup_{\u,\v\in S^{n-1}}\sqrt{\E_\mathcal{E}\big[(\v^\top\Sigmab^{-1/2}\Pperp{k}\x_k\x_k^\top\Sigmab^{-1/2}\u)^2\big]}. \label{eq:cauchy-schwartz}
\end{align}

We estimate these two multiples separately. The first term is small by Euclidean concentration, and the second term is bounded (almost) recursively similar to the bound on $T_2$.

For the first term in \eqref{eq:cauchy-schwartz}, we have $\xi = \x_k^\top\Pperp{k}\x_k\geq \frac12\tr(\Sigmab\Pperp{k})$ due to conditioning, so
\begin{equation*}
\E_\mathcal{E}
        \left[\frac{|\E\xi - \xi|^2}{(\x_k^\top\Pperp{k}\x_k)^2}\right]
       \leq \frac{4}{\PP(\EE)}\cdot
        \E\left[\frac{|\E\xi - \xi|^2}{\tr(\Sigmab\Pperp{k})^2}\right] \leq \frac{4}{1-\delta}\cdot \E\Bigg[\frac{\Var\big[\x_k^\top\Pperp{k}\x_k\mid\P_{-k}\big]}{(\tr\,\Sigmab\P_{\perp k})^2}\Bigg].
\end{equation*}
The Euclidean concentration property of $\s_k$ implies the following bound on the variance of the quadratic form $\x_k^\top\Pperp{k}\x_k=\s_k^\top\A\Pperp{k}\A^\top\s_k$ (Lemma~\ref{lem:concentration-sub-gaussian} in the appendix):
 \begin{equation*}
\Var\big[\x_k^\top\Pperp{k}\x_k\mid\P_{-k}\big]
    \leq CL^2\|\A\Pperp{k}\|^2(L^2\|\A\Pperp{k}\|^2 + \|\A\Pperp{k}\|_F^2)
    \leq CL^4 + CL^2\tr\,\Sigmab\Pperp{k},
    \end{equation*}
so using that $\tr\Sigmab\Pperp{k}\geq r-k$ by \eqref{trace-key}, we obtain:
\begin{equation*}
   \E_\mathcal{E}
        \Bigg[\bigg(\frac{\E \xi- \xi}{\xi}\bigg)^2\Bigg] \leq
        8C\Big(\frac{L^4}{(r-k)^2}+\frac{L^2}{r-k}\Big)\leq \frac{16CL^2}{r-k}. 
\end{equation*}
For the second term in \eqref{eq:cauchy-schwartz}, recall that $\x_k = \A^\top \s_k$, so, using Cauchy-Schwartz, we have 
\begin{align*}
    &\E_\EE\big[(\v^\top\Sigmab^{-1/2}\Pperp{k}\x_k \cdot
    \x_k^\top\Sigmab^{-1/2}\u)^2\big]
    \\
    &\leq \frac{1}{\PP(\Ec)}\E\bigg[\sqrt{\E\big[(\v^\top\Sigmab^{-1/2}\Pperp{k}\A^{\top} \s_k)^4\mid \P_{-k}\big]} \cdot \sqrt{\E\big[(\s_k^\top\A\Sigmab^{-1/2}\u)^4\mid \P_{-k}\big]}
    \bigg]
    \\
    &\leq  
    2\E\bigg[\sqrt{CK^4\|\A\Pperp{k}\Sigmab^{-1/2}\v\|^4} \cdot \sqrt{CK^4\|\A\Sigmab^{-1/2}\u\|^4}\bigg]
  \leq C'K^4\E\Big[\v^\top\Sigmab^{-1/2}\Pperp{k}\Sigmab\Pperp{k}\Sigmab^{-1/2}\v\Big]
    \end{align*}
    where we used $K$-sub-gaussianity of $\s_k$ and the fact that $\|\A\Sigmab^{-1/2}\u\|=\|\u\|=1$.
Thus, combining the above with the bound on $\Pperp{k}\Sigmab\Pperp{k}$ derived earlier in \eqref{third_one},
    \begin{align*}
    \sup_{\u,\v\in S^{n-1}}&\sqrt{\E_\mathcal{E}\big[(\v^\top\Sigmab^{-1/2}\Pperp{k}\x_k\x_k^\top\Sigmab^{-1/2}\u)^2\big]}
    \leq \sup_{\v\in S^{n-1}} \sqrt{CK^4\E\Big[\v^\top\Sigmab^{-1/2}\Pperp{k}\Sigmab
    \Pperp{k}\Sigmab^{-1/2}\v\Big]}
    \\
    &\leq K^2\sqrt{C\big\|\Sigmab^{-1/2}\E\big[\Pperp{k}\Sigmab\Pperp{k}\big]\Sigmab^{-1/2}\big\|}\overset{\eqref{third_one}}{\leq}     K^2\sqrt{2C(1+\|\overbar\P^{-1/2}\E[\P]\overbar\P^{-1/2}\|)}.
\end{align*}
As a result, again adjusting $C$, and using $\epsilon_1 = \|\overbar\P^{-1/2}\E[\P]\overbar\P^{-1/2}\|$, we conclude that:
\begin{equation*} 
 T_3 \leq C \sqrt{\frac{L^2}{r-k}}\cdot  K^2\sqrt{C(1+\|\overbar\P^{-1/2}\E[\P]\overbar\P^{-1/2}\|)}
 \leq \frac{CLK^2}{\sqrt{r}}\cdot\sqrt{1+\epsilon_1}.
\end{equation*} 

\textbf{5. Conclusion via a recursive bound.}  So far, we bounded $\epsilon_\Ec$ by a quantity $\epsilon_1$ that itself depends on $\epsilon$, which seems (almost) circular:
\begin{equation}\label{rhoeps}
\epsilon_\Ec \le \kappa^2\delta(1 + C_1L^2 r) +\frac{C_2}{\sqrt{r}} \sqrt{1 + \epsilon_1} + \frac{C_3LK^2}{\sqrt{r}}\cdot  \sqrt{1 + \epsilon_1}.
\end{equation}
However, together with \eqref{rho-rhoeps}, we have the following relation
\begin{align*}
\epsilon_1 &\le 1 + \epsilon \le 1 + \epsilon_\Ec + 2(1 + r \kappa^2)\delta 
\\&
\le 1  + 2(1 + r \kappa^2)\delta + \kappa^2\delta(1 + CL^2 r) + \frac{CLK^2}{\sqrt{r}}\cdot  \sqrt{1 + \epsilon_1} 
\le \frac{3}{2} + \frac{1}{2}\sqrt{1+ \epsilon_1},
\end{align*}
where we used that $r\geq C'\max\{k,L^2\log \kappa,K^4L^2\}$  for large $C' > 0$ and $\delta=2\exp(-cr/L^2)$.

Solving for $\epsilon_1$, we see that $\epsilon_1 \le 3$. Plugging it back to \eqref{rhoeps}, we get that $\epsilon_\Ec \le C L K^2/\sqrt{r}$, and through \eqref{rho-rhoeps} we have that $\epsilon \le CLK^2/\sqrt{r}$ with a new constant $C $.  This concludes the proof of Theorem~\ref{t:sub-gaussian} by \eqref{t2-sym}.
\end{proof}

\subsection{Sparse sketching: proof of Theorem~\ref{c:less}}\label{s:less-proof}

We will use two results from prior work. The first one shows that a LESS embedding is close in total variation distance to an embedding with strong sub-gaussian concentration. Recall that for random variables $X$ and $Y$, their total variation distance $d_{\textnormal{tv}}(X,Y)$ is the infimum over $\delta>0$ such that there exists a coupling between $X$ and $Y$ such that $\PP(X\neq Y)=\delta$.
\begin{lemma}[\cite{gaussianization}]\label{l:tv-surrogate}
Let $\delta \in (0,1)$. Consider an $m\times n$ matrix $\A$, where $m\geq n$, and let $\S$ be a LESS embedding matrix for $\A$ with $k$ rows having $O(n\log(nk/\delta))$ non-zeros per row. There is a positive semidefinite matrix $\Sigmabt$ and a $k\times n$ random matrix $\Z$ with i.i.d. isotropic rows that satisfy both sub-gaussianity and Euclidean concentration with absolute constants, such that:
    \begin{equation}\label{less-sub}
       d_{\textnormal{tv}}(\S\A,\Z\Sigmabt^{1/2})\leq\delta
        \qquad\text{and}\qquad  (1-\delta)\A^\top\A\preceq\Sigmabt\preceq(1+\delta)\A^\top\A.
    \end{equation}
\end{lemma}
The second result, which is an immediate corollary of \cite{precise-expressions}, gives a precise analytic expression for the expected Randomized SVD error in terms of the singular values of matrix $\A$.
\begin{lemma}[\cite{precise-expressions}]\label{p:implicit}
Under the assumptions of Theorem \ref{t:sub-gaussian}, the expected error of Randomized SVD can be approximated by the following implicit analytic expression: 
\begin{equation*}
    (1-\epsilon)\cdot k/\bar\gamma_k\leq
    \Err(\A,k-1)\leq(1+\epsilon)\cdot k/\bar\gamma_k\quad\text{for}\quad\epsilon=O(1/\sqrt r),
\end{equation*} 
where $\bar\gamma_k = f^{-1}(k)$, i.e., it is the function inverse of $f$ at $k$, for $f(\gamma)=\tr\,\gamma\Sigmab(\gamma\Sigmab+\I)^{-1}$.
\end{lemma}  

We now proceed with the proof of Theorem~\ref{c:less}, by comparing the expected projection based on $\S\A$ with an analogous expected projection matrix defined by $\Z\Sigmabt^{1/2}$ (Lemma~\ref{l:tv-surrogate}).

\begin{proof}[Proof of Theorem~\ref{c:less}]
We substitute our LESS-embedding projection  $\P$ with a projection $\tilde\P$ defined via Lemma~\ref{l:tv-surrogate} as: \begin{equation*}\tilde\P:=(\Z\Sigmabt^{1/2})^\dagger\Z\Sigmabt^{1/2}.\end{equation*} First, note that $\tilde\P$ satisfies assumptions of Theorem \ref{t:sub-gaussian}: indeed, $\Z$ satisfies the assumptions on the sketching matrix by construction, and both the condition number $\kappa$ and the stable rank $r$ of $\Sigmabt^{1/2}$ are within a factor of $(1\pm\delta)$ of the corresponding quantities for $\A$ by \eqref{less-sub}. So,  they satisfy the required mutual relations, potentially with slightly different absolute constants. Thus, combining the estimates from Theorem~\ref{t:sub-gaussian} and  Lemma~\ref{p:implicit}, we have 
\begin{equation}\label{less-approx-thm2}
    (1-\tilde\epsilon)\tilde\gamma_k\Sigmabt(\tilde\gamma_k\Sigmabt+\I)^{-1}\preceq    \E[\tilde\P]\preceq(1+\tilde\epsilon)\tilde\gamma_k\Sigmabt(\tilde\gamma_k\Sigmabt+\I)^{-1}\quad\text{for}\quad\tilde\epsilon=O(1/\sqrt r),
\end{equation}
where $\tilde\gamma_k =  \tilde f^{-1}(k)$, with $\tilde f(\gamma)=\tr\,\gamma\Sigmabt(\gamma\Sigmabt+\I)^{-1}$. Our goal is to show that $\E[\P]$ and its respective surrogate expression  $\gamma_k\Sigmab(\gamma_k\Sigmab+\I)^{-1}$ satisfy the same mutual relationship \eqref{less-approx-thm2} with $\epsilon \leq C \tilde\epsilon$. We will do this in two steps: first, we estimate $\E[\P]$ via new surrogates coming from \eqref{less-approx-thm2}, and then, we show that the two surrogate expressions are close. 

\vspace{0.15cm}

\textbf{Step 1.}  Since $\S\A$ and $\Z\Sigmabt^{1/2}$ are within $\delta$ total variation distance, we can couple these two random matrices so that the event $\Ec := \{\P = \tilde \P\}$ holds with probability $1-\delta$.  We observe that since $\tilde\gamma_k\geq k/\tr(\Sigmabt)\geq k/n$, we have $\lambda_{\min}(\tilde\gamma_k\Sigmabt(\tilde\gamma_k\Sigmabt+\I)^{-1})=\lambda_{\max}^{-1}(\I + \tilde\gamma_k^{-1}\Sigmabt^{-1}) \geq k/(2n\kappa^2)$. As a result, letting $\delta\leq \epsilon/(n\kappa^2)$, using the law of total probability and that $\zero\preceq\P\preceq\I$,
\begin{align*}
    \E[\P] &= \PP(\Ec)\E[\P\mid\Ec] + \PP(\neg\Ec)\E[\P\mid\neg\Ec]
\succeq \PP(\Ec)\E[\tilde\P\mid\Ec]
    = \E[\tilde\P] - \PP(\neg\Ec)\E[\tilde\P\mid\neg\Ec]  
    \\
&\succeq\E[\tilde\P] - \delta\I
  \succeq\E[\tilde\P]-\frac{\epsilon}{n\kappa^2}\cdot 2n\kappa^2\,\tilde\gamma_k\Sigmabt(\tilde\gamma_k\Sigmabt+\I)^{-1}
\succeq (1-3\epsilon)\tilde\gamma_k\Sigmabt(\tilde\gamma_k\Sigmabt+\I)^{-1},
   \end{align*}
    and on the other side, analogously we obtain:
\begin{equation*} 
    \E[\P] = \PP(\Ec)\E[\P\mid\Ec] + \PP(\neg\Ec)\E[\P\mid\neg\Ec]
    \preceq \PP(\Ec)\E[\tilde\P\mid\Ec] + \delta\I
    \preceq (1+3\epsilon)\tilde\gamma_k\Sigmabt(\tilde\gamma_k\Sigmabt+\I)^{-1}.
\end{equation*} 
This way, we estimated $\E[\P]$ via the approximated surrogate $\tilde\gamma_k\Sigmabt(\tilde\gamma_k\Sigmabt+\I)^{-1}$.

\vspace{0.15cm}

\textbf{Step 2.} It remains to show that $\tilde\gamma_k\Sigmabt(\tilde\gamma_k\Sigmabt+\I)^{-1}$ is close to  $\bar\gamma_k\Sigmab(\bar\gamma_k\Sigmab+\I)^{-1}$, for $\bar\gamma_k=f^{-1}(k)$, with $f(\gamma)=\tr \gamma\Sigmab(\gamma\Sigmab+\I)^{-1}$. After that, we can conclude Theorem~\ref{c:less} with $\gamma_k = k/\Err(\A,k-1)$ by using Lemma~\ref{p:implicit} again.

Further, it is enough to compare $\tilde\gamma_k$ and $\bar\gamma_k$ since we know that $\Sigmabt$ is a $(1\pm\delta)$-approximation of $\Sigmab$ \eqref{less-sub}. Another consequence of $\Sigmabt$ and $\Sigmab$ being $\delta$-close is that $\tilde\gamma_k = f^{-1}(\tilde k)$ for some $\tilde k$ being a $(1\pm\delta)$-approximation of $k$. Thus, it suffices to show that $f^{-1}(\tilde k)\approx f^{-1}(k)$ for
\begin{equation*}f(\gamma)=\tr\, \gamma\Sigmab(\gamma\Sigmab+\I)^{-1} \quad\text{ and any } \tilde k: \; |k - \tilde k| \le \delta k. \end{equation*}

We show this by bounding the derivative of $f^{-1}$ at any point in $[0,r/2]$ (which includes both $k$ and $\tilde k$), and then relying on the mean value theorem. Let $\sigma_1,\sigma_2,...$ be the decreasing singular values of $\A$, and assume without loss of generality that $\sigma_1=1$. Observe that 
\begin{equation*}
f(\gamma) 
    = \sum_i\frac{\gamma\sigma_i^2}{\gamma\sigma_i^2+1} \ge \frac{\gamma \sum_i\sigma_i^2}{\gamma\sigma_1^2+1} = \frac{\gamma r}{\gamma + 1},
\end{equation*}
    so $f^{-1}(x) = \gamma \le \frac{x}{r -x} \le 1$ for $x \le r/2$.
    This holds for $k$, $\tilde k$ and any value in between, so 
\begin{equation*}
    f'(\gamma) 
    = \frac{\text{d}}{\text{d}\gamma}
    \Big(\sum_i\frac{\gamma\sigma_i^2}{\gamma\sigma_i^2+1}\Big)
    =\sum_i \frac{\sigma_i^2}{(\gamma\sigma_i^2+1)^2}\geq \frac14\sum_i\sigma_i^2=\frac{r}{4}
\;
\text{ and }
\;
    \frac{\text{d}f^{-1}(x)}{\text{d}x} = \frac1{f'(f^{-1}(x))}\leq \frac{4}{r}.
\end{equation*} 
Now, by the mean value theorem:
\begin{equation*}
    |\tilde\gamma_k - \bar\gamma_k|=|f^{-1}(\tilde k) - f^{-1}(k)| \leq \frac 4r\cdot|k-\tilde k|\leq 4\delta\cdot\frac{k}{r}\leq 4\delta\cdot\bar\gamma_k,
\end{equation*}
where the last step follows because 
\begin{equation*}k=\sum_i\frac{\bar\gamma_k\sigma_i^2}{\bar\gamma_k\sigma_i^2+1}\leq \bar\gamma_k\sum_i\sigma_i^2=\bar\gamma_k r.\end{equation*}
Thus, $\tilde\gamma_k$ and $\bar\gamma_k$ are within $1+O(\delta)$ factor to each other, and so the two surrogate expressions are similarly close. We conclude by using $\epsilon = 2 n \kappa^2\delta$ and $O(n \log(n k /\delta))$ non-zeros per row.
\end{proof}

\section{Experiments}\label{s:experiments}
In this section, our goal is to provide additional support for our theoretical bounds and to demonstrate how the new bounds can inform us about the optimal sketch sizes and sparsity levels when using sketch-and-project methods.
In Sections \ref{subsec:experiments-size} and \ref{subsec:experiments-bounds}, we test the Block Gaussian Kaczmarz method (assuming $\B= \I$ throughout) with various sketch sizes $k$. Then, in Section \ref{subsec:experiments-sparse}, we replace the Gaussian sketching matrix with LESS-based sparse sketching matrices, varying their sparsity. We consider artificial linear problems, as well as matrices coming from real-world datasets. 

The artificial matrices have dimensions $5000 \times 150$, and are generated to have specific spectral profiles. Namely, models 'lin.01', 'lin.025' and 'lin.035' have singular values $\sigma_i = 6.8-l\cdot i$ with $l = 0.01, 0.025$, and $0.035$ (which results in $\sigma_{\min}$ being $5.3$, $3.1$ and $1.6$, respectively); Models 'poly1' and 'poly1.5' are generated so that $\sigma_i = 6.8\,i^{-l}$, for $l = 1$ and $l = 1.5$ respectively. Finally, in 'step20' model, $20$ top singular values are as in 'lin.01' and the rest are as in 'poly1', and 'step37' is defined analogously.
We also consider three other matrices: a Gaussian matrix with normalized rows ('gaus'); a sub-matrix of the USPS dataset \cite{chang2011libsvm} ('usps'); and a sub-matrix of the w8a-X dataset \cite{chang2011libsvm} ('w8a'). Further details on data generation are in Section SM4 of the appendix.

\begin{figure}
\centering
\includegraphics[width=\linewidth]{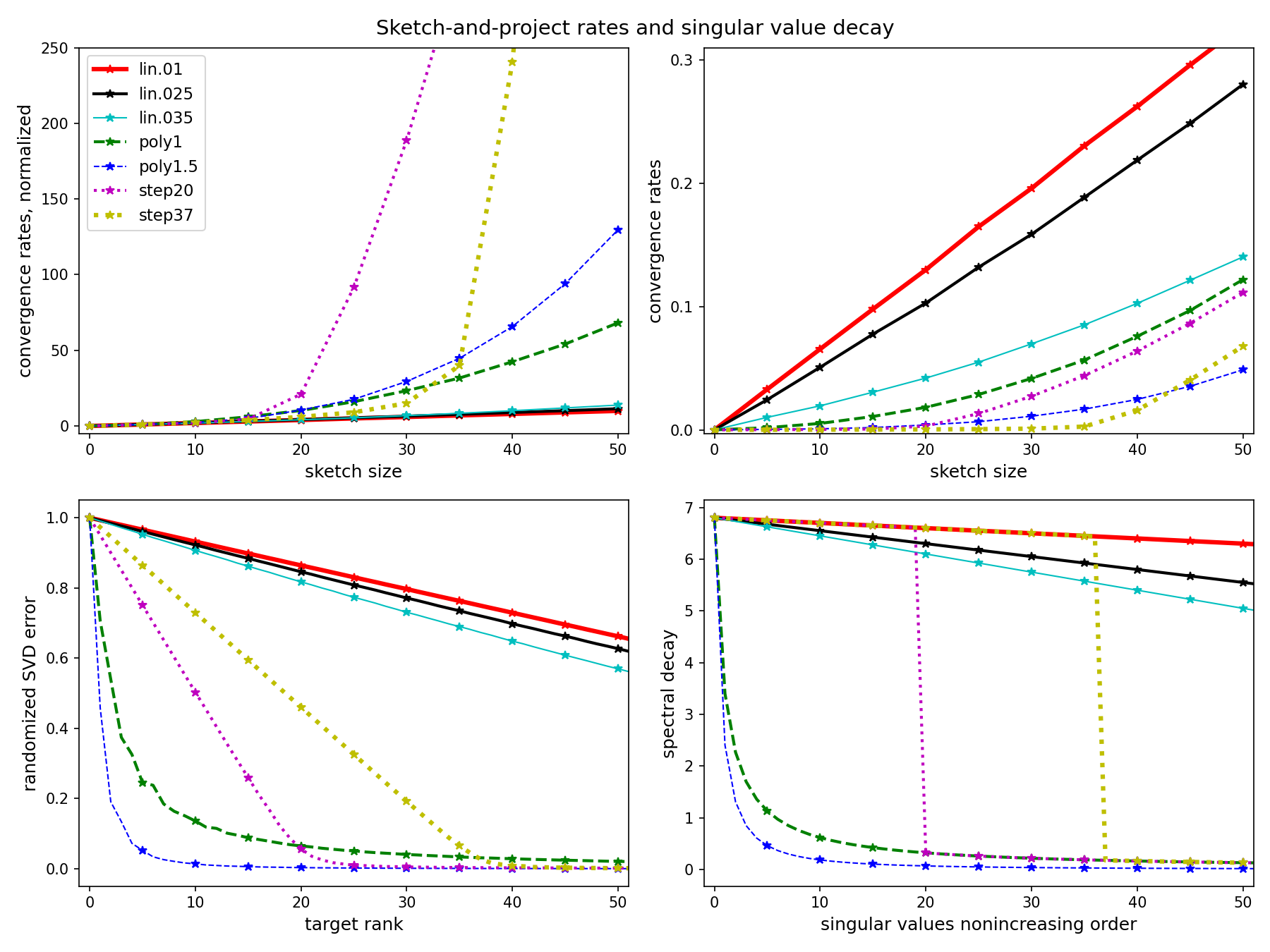}
\vspace{-5mm}
\caption{Two figures on the top show the estimated convergence rates vs sketch size (larger is better),
computed as the mean of 
$1 - \frac{\|\x_{t} - \x_*\|^2}{\|\x_{t-1} - \x_*\|^2} \sim \Rate(\A,k)$
 over the last $50$ iterations in $100$ runs, to show the ``limiting" rate after stabilization. We run the method for $1000$ iterations or until $\|\x_{t} - \x_*\| < 10^{-5}$. We estimate the expected rate for the sketch sizes $k = 5j$ for $j =1, \ldots, 10$. Top right plot shows that linear decays result in faster convergence in general due to having larger smallest singular value. However, the normalized version of the same plot on the left demonstrates that the relative speed-up with increased sketch size is quicker for polynomial spectral decays, shadowing the claim of Corollary~\ref{cor:poly-decays} (polynomial rate growth on polynomial spectra) and Corollary~\ref{cor:proj-control} (inverse to the profiles of Randomized SVD error). Normalized Randomized SVD error is computed as $\|\A(\I - \P)\|_F/\|\A\|_F$ (top right, smaller~is~better).}
\label{fig-decays}
\end{figure}
 \begin{figure}
\centering
\includegraphics[width=\linewidth]{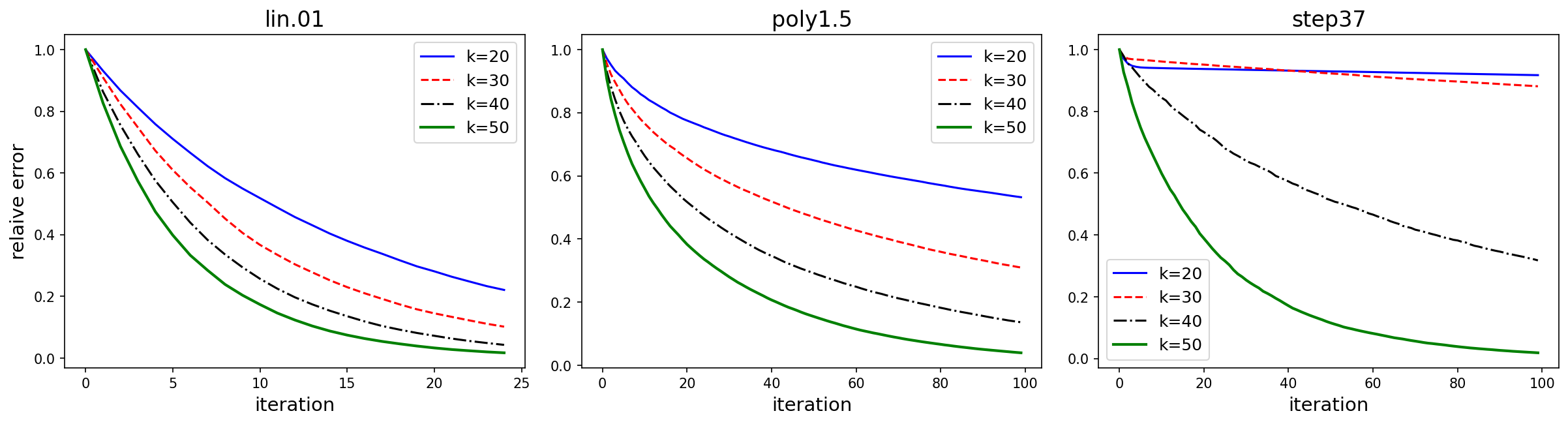}
\vspace{-2mm}
\caption{Per iteration convergence for various artificial models and various sketch sizes $k$. We plot the relative error $\|\x_t - \x_*\|/\|\x_*\|$ after $t$ iterations, with $\x_0$ being the all zeros vector. Polynomial model results in substantial improvements as sketch size grows, and the linear model convergence improves less with the growing sketch size. For the step singular values decay, we see a threshold: having the sketch size bigger than the number of ''large" singular values (approximate rank) drives an effective convergence~rate.}
\label{fig-real-world-1}
\end{figure}

\subsection{Sketch size and convergence}
\label{subsec:experiments-size}

In Figure~\ref{fig-decays}, we demonstrate the scaling of the estimated convergence rate with the sketch size $k$. In particular, we show that linear spectral decays result in linear per-iteration speed up with growing sketch size, whereas quicker (polynomial) spectral decays result in accelerated rate improvement with sketch size. Moreover, we demonstrate the similarity between the error of Randomized SVD \eqref{randsvd_error} and the convergence speed-up of sketch-and-project: when the spectral decay has an abrupt change at $i$-th ordered singular value, then  about $i$ components are needed for a small error in Randomized SVD, and similarly, sketch sizes $k\ge i$ drive efficient convergence of the iterative solver. 

In Figure~\ref{fig-decays} (top right), we see that linear models have higher rates and they are generally the fastest to converge. This is not surprising as they are the best conditioned (having smallest singular values above $1.5$). To emphasize the relationship between the rate and the sketch size, in Figure~\ref{fig-decays} (top left), we normalize the rates by the first non-trivial point, the rate at $k = 5$. We can see that the acceleration of the rate for linear models is the slowest, and that 'step' models spike with the behavior roughly reverse to their Randomized SVD error, plotted in  Figure~\ref{fig-decays} (bottom left). In Figure~\ref{fig-real-world-1}, we demonstrated that faster polynomial spectral decay means better improvement of the convergence speed with growing $k$.

\subsection{Sharpness of our theoretical bounds}
\label{subsec:experiments-bounds}
In Figure~\ref{fig-surrogates}, we compare the estimated convergence rate (same as in Figure~\ref{fig-decays}) and its theoretical lower bounds: (a) 's-min': the worst-case expected convergence rate $\lambda_{\min}(\E[\P]) = \Rate_\B(\A,k)$, and (b) 'surrogate': our new surrogate estimate for $\lambda_{\min}(\E[\P])$, as per Corollary~\ref{cor:proj-control}, defined via the Randomized SVD error: 
\begin{equation}\label{surrogate-bound}\frac{k\sigma_{\min}^2(\A)}{k\sigma_{\min}^2(\A)+\E\,\|\A(\I-\P^{(k-1)})\|_F^2},
\end{equation}
where $\P^{(k-1)}=(\S^{(k-1)}\A)^{\dagger}\S^{(k-1)}\A$. We see that the two theoretical bounds 's-min' and 'surrogate' are very close to each other (in most cases, they are difficult to tell apart), which confirms that $\epsilon$ from Theorem~\ref{t:sub-gaussian} is very close to $0$ in practice. Furthermore, both lower bounds effectively mimic the shape of the empirical convergence rate curve, as we vary the sketch size $k$. We note that the remaining gap between the observed rate and the worst-case expected convergence rate is due to the fact that real iterations $\x_k$ will not all be in the worst-case position $\x_k = \x_* + \alpha\v$, where $\v$ is the smallest singular vector of the data.

\begin{figure}
\centering
\includegraphics[width=\linewidth]{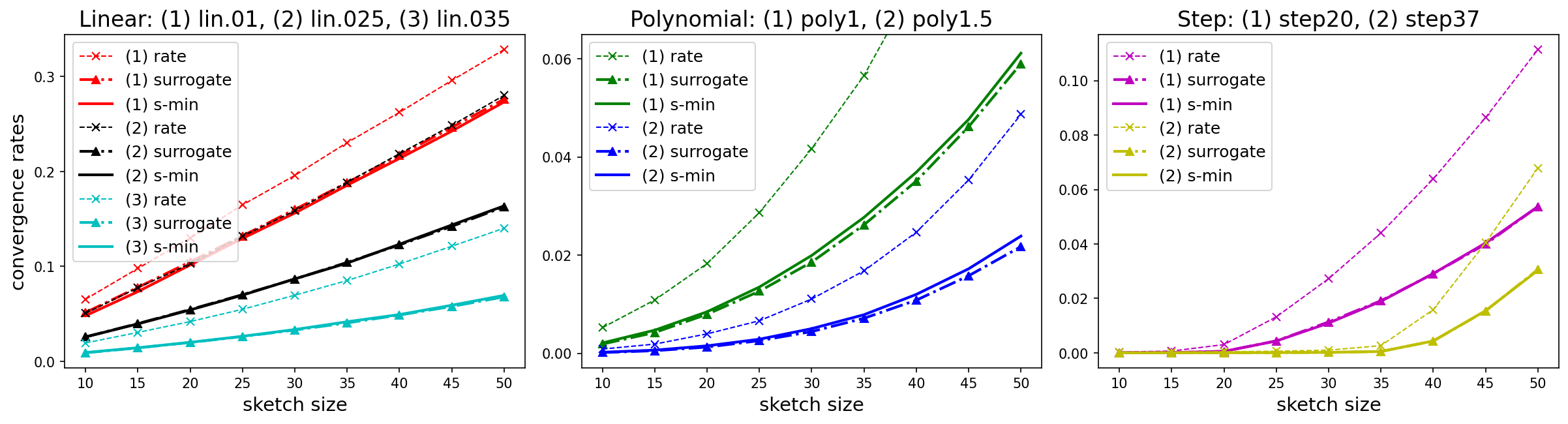}
\caption{As claimed by Theorem~\ref{t:sub-gaussian}, our surrogate expression, computed as in \eqref{surrogate-bound} (where the expectation is estimated as a mean over $50$ Gaussian $(k-1) \times m$ sketches $\S^{(k-1)}$), is a very tight estimator for s-min := $\lambda_{\min} (\E[\P])$ (plotted here as the smallest eigenvalue of the mean of $1600$ sampled projection matrices). Both quantities lower bound and mimic the convergence rate of sketch-and-project (larger rate is better).}
\label{fig-surrogates}
\end{figure}

\subsection{Sparse sketching and convergence}
\label{subsec:experiments-sparse}
\begin{figure}
\centering
\includegraphics[width=\linewidth]{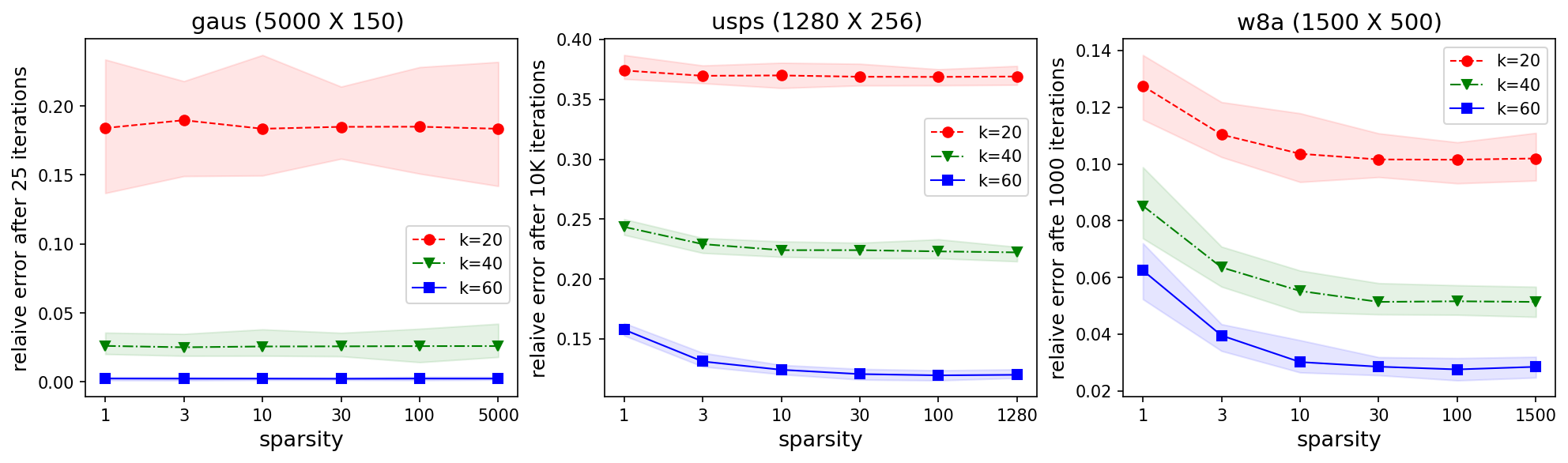}
\caption{Relative error after the same number of iterations for dense and sparse sketching matrices $\S$, averaged over $30$ runs (smaller is better). Regions are shaded between maximum and minimum error in $30$ runs. Sparsity is the number of non-zeros per row of $\S$: the rightmost point refers to the dense case and the leftmost point is for extreme sparsity. We observe consistency of performance between dense and significantly sparse matrices $\S$. For Gaussian data, a $1$-sparse~$\S$ (that effectively does row sampling) performs identically to a dense $\S$. However, depending on the dataset, extremely sparse~$\S$ can lead to slower convergence.}
\label{fig-sparsity}
\end{figure}

In Figure~\ref{fig-sparsity}, we show that sketches formed with extremely sparse matrices $\S$ often exhibit similar convergence behavior to a dense Gaussian~$\S$, confirming our  theoretical results obtained in Theorems~\ref{t:sub-gaussian} and \ref{c:less}. We use uniformly sparsified sketching matrices (i.e., without approximating the leverage scores), which corresponds to the strategy suggested in Remark~\ref{r:less} part 2. In theory, to recover our convergence guarantees for these sparse sketches, we would have to first precondition the linear system using a randomized Fourier transform. However, as is common in practice, we omit this preconditioning here.
We also note that $1$-sparse sketching matrices are equivalent to block sketching via row sampling since multiplying by a Gaussian scalar does not change the subspace associated with a particular equation hyperplane. In the Kaczmarz case, it is the same as a randomized block Kaczmarz version with randomly sampled rows (as discussed in \cite{haddock2021greed}). Even though the $1$-sparse case is not covered by our theory, we show that general trends, such as scaling with the sketch size, can be estimated for cost-efficient block sketching from the theoretical rates.

In our experiments, the sparsity level is controlled by the number of non-zeros per row of sketching matrix~$\S$ (smaller number means sparser matrix). The non-zero entries of $\S$ are Gaussian, and chosen uniformly at random, which corresponds to a simplified version of LESS embeddings, called LessUniform \cite{newton-less}. For a highly coherent dataset, one might consider pre-processing the data with a randomized Fourier transform, as discussed in Section~\ref{s:main-sparse}. Clearly, sparse sketching matrices can be stored and applied much faster than dense ones: a $k\times m$ sketching matrix $\S$ with $s$ non-zeros per row requires $O(ks)$ memory and $O(ksn)$ time to apply to an $m\times n$ matrix $\A$. 

Specifically, we compare the relative error after $30$ steps of the algorithm for several sketch sizes and sparsity regimes. We observe that even sketches with extremely sparse matrix $\S$ can lead to a convergence that is nearly as fast as in the dense Gaussian case. Only for some real-world datasets, we observe some degradation in performance as we get close to $1$-sparse sketching matrices (i.e., one non-zero per row). This suggests that the $O(n\log n)$ non-zeros per row required by our theory (Theorem~\ref{c:less}) is very conservative. 

\section{Conclusions and future directions}
In this work, we give new sharper bounds for the convergence rate of iterative solvers based on the sketch-and-project framework, and also provide first convergence guarantees for structured and sparse sketches. Our main technical contribution is providing a new spectral analysis of expected projections onto the span of random matrices that arise from sketching. An important byproduct of our approach is a discovered connection between two classical and seemingly very different applications of sketching: the convergence rate of sketch-and-project and the approximation error of Randomized~SVD.

Further extensions of our analysis could include quantifying the initial phase of the accelerated convergence, as well as a sharp analysis of sketch-and-project for noisy or under-determined linear systems. We also conjecture that our assumptions on the stable rank and condition number in Theorem~\ref{t:sub-gaussian} (full-spectrum analysis of the expected projection) can be relaxed, and we leave this as an open question for future work. 

More broadly, our results give rise to many potential future directions. First, given the new connection between sketch-and-project and Randomized SVD, it is natural to ask whether recent algorithmic developments in sketching-based matrix approximation (see \cite{martinsson2020randomized} for an overview) can inspire new iterative solvers and extensions of sketch-and-project. Furthermore, expected projections also find applications in machine learning and statistics, for example in the analysis of random feature models \cite{bach2017equivalence}, density estimation \cite{smola2007hilbert}, Gaussian processes \cite{sparse-variational-gp}, kernel mean embeddings \cite{muandet2017kernel}, numerical integration algorithms \cite{belhadji2019kernel}, and more. Extending our results to these applications is a promising direction for future~work.

\subsection{Acknowledgments}
The authors are grateful to Michael Mahoney and Matias Cattaneo for valuable discussions and feedback. E.R. was partially supported by NSF DMS-2309685 and DMS-2108479.

\bibliographystyle{abbrv}
\bibliography{pap,references}

\appendix
\section{Auxiliary results for the proof of Theorem~\ref{t:gaussian}}\label{sec:app_a}

\begin{proof}[Proof of Lemma~\ref{l:tensor-function}]
We start with the following preliminary consideration: note that a matrix function $\M: \mathrm{Sym}(n) \to \mathrm{Sym}(n)$ defined as $$\M(\Sigmab):=\E[\Sigmab^{1/2}\Z^\top(\Z\Sigmab\Z^\top)^\dagger\Z\Sigmab^{1/2}]$$ is an \emph{isotropic tensor function}, i.e., that $\M(\Q\Sigmab\Q^\top)=\Q\M(\Sigmab)\Q^\top$ for all $n\times n$ orthogonal matrices $\Q$. This follows because, due to rotational invariance of the Gaussian distribution, random matrix $\Z\Q$ is distributed identically as $\Z$ for any orthogonal $\Q$, so we have for any symmetric positive semidefinite matrix $\Sigmab$:
\begin{align*}
    \M(\Q\Sigmab\Q^\top) &= \E[\Q\Sigmab^{1/2}\Q^\top\Z^\top(\Z\Q\Sigmab\Q^\top\Z^\top)^\dagger\Z\Q\Sigmab^{1/2}\Q^\top]
    \\
    &=\Q\E[\Sigmab^{1/2}(\Z\Q)^\top(\Z\Q\Sigmab(\Z\Q)^\top)^\dagger\Z\Q\Sigmab^{1/2}]\Q
    =\Q\M(\Sigmab)\Q^\top.
\end{align*}
Additionally, from the literature on continuum mechanics 
(see Section 6.4 in \cite{itskov2007tensor}; e.g., Theorem 6.2), it is known that any isotropic tensor function  $\M(\Sigmab)$ has the same eigenbasis as $\Sigmab$. Equipped with that, we are ready to prove the three claims of the lemma:

1.\quad Note that 
        \begin{equation}\label{ep-m}
    \E[\P] =  \E[(\Z\D\V^\top)^\dagger\Z\D\V^\top] = \M(\V\D^2\V^\top) = \V\M(\D^2)\V^\top=\V\E[\tilde\P]\V^\top,
    \end{equation}
    since the matrices $\S\U$ and $\Z$ are equidistributed due to rotational invariance of the Gaussian distribution, and then we use the isotropic tensor property of $\M$. The claim follows since spectrum is preserved under orthogonal rotations.

2.\quad The matrix $\D$ has the same (standard) basis as $\D^2$, which is the same as the basis of $\M(\D^2)$ due to the isotropic tensor property. 

3.\quad Due to circular trace property, linearity of expectation and \eqref{ep-m}, we can rewrite
\begin{align*}
    \Err(\A, k) = \tr(\A^\top\A(\I - \E[\P])) & = \tr(\V\D^2 \V^T( \I - \M(\V \D^2 \V^\top)))  \\
    & = \tr(\V\D^2 \V^T( \I - \V \M(\D^2) \V^\top) \\
    & = \tr(\D^2(\I - \M(\D^2)) = \Err(\D, k).
\end{align*}
\noindent
This concludes the proof of Lemma~\ref{l:tensor-function}.
\end{proof}

\begin{lemma}[Sherman-Morrison formula]\label{lem:rank-one}
For an invertible matrix $\A \in \R^{n \times n}$ and $\u,\v \in \R^n$, $\A + \u \v^\top$ is invertible if and only if $1+\v^\top \A^{-1} \u \neq 0$. If this holds, then
\begin{equation*}
    (\A + \u \v^\top)^{-1} = \A^{-1} - \frac{\A^{-1} \u \v^\top \A^{-1} }{1+\v^\top \A^{-1} \u}.
\end{equation*}
In particular, it follows that: 
\begin{equation*}
    (\A + \u \v^\top)^{-1} \u = \frac{\A^{-1} \u}{1+\v^\top \A^{-1} \u}.
\end{equation*}
\end{lemma}

The next lemma is essentially a sharper version of the Hanson-Wright concentration inequality for Gaussian vectors.
\begin{lemma}\label{lem:hw_tail}
For any fixed $n\times n$ positive semidefinite matrix $\Y$, an $n$-dimensional standard Gaussian vector $\s$, and any $\alpha > 0$, let
$$
\EE_{\alpha} := \{\s^\top\Y\s \leq \tr\,\Y + \sqrt{2\alpha\,\tr\,(\Y^2)}+\alpha\,\|\Y\|\}.
$$
Then
  \begin{align}\label{hw_gaus}
 \mathbb{P}\big[\neg\EE_{\alpha}\big] \leq \exp(-\alpha/2) \quad 
\text{ and } \quad 
 \E[\s^\top\Y\s\cdot\one_{\neg\EE_\alpha}] \le 5 \tr\,(\Y)\cdot e^{-\alpha/2}.
 \end{align}
\end{lemma}
\begin{proof}
  The first claim of \eqref{hw_gaus} is  a sharper version of Hanson-Wright inequality for Gaussian vectors, proved in \cite{hsu2012tail}. So, $\E[\one_{\neg\EE_\alpha}]=\PP(\neg\EE_\alpha)\leq e^{-\alpha/2}$. For the second claim, letting $u_\alpha:=\sqrt{2\alpha\,\tr\,(\Y^2)}+\alpha\,\|\Y\|$, we have:
  \begin{align*}
  \E[(\s^\top\Y\s) \one_{\neg\EE_\alpha}] 
    &\le \tr\,\Y\cdot e^{-\alpha/2} + \E[(\s^\top\Y\s-\tr\,\Y) \one_{\neg\EE_\alpha}]\\
    &=\tr\,\Y\cdot e^{-\alpha/2} + \int_{u_\alpha}^\infty \PP\big[\s^\top\Y\s \geq \tr\,\Y+t\big]\, dt
    \\
    &\leq
      \tr\,\Y\cdot e^{-\alpha/2} + \int_{u_\alpha}^\infty \exp\bigg(-\min\Big\{\frac{t^2}{4\,\tr\,(\Y^{2})},\frac{t}{2\|\Y\|}\Big\}\bigg)\,dt
    \\
    & \leq \tr\,\Y\cdot e^{-\alpha/2} + \sqrt{4\,\tr\,(\Y^{2})}\cdot e^{-\alpha/2} + 2\|\Y\|\cdot e^{-\alpha/2}\\
    &\leq 5\,\tr\,\Y\cdot e^{-\alpha/2}.
  \end{align*}  
  This concludes the proof of the lemma.
  
\end{proof}

Finally, the proof of Theorem~\ref{t:gaussian} is concluded by the following computational lemma:
\begin{lemma}\label{lem:gaus-compute}
Let $\Y = (\S\D^2\S^\top)^{-1}$ where $\S$ is a $k\times n$ Gaussian matrix, and $\D$ is  $n\times n$ diagonal matrix positive eigenvalues. Let $\mathcal{U}_{\alpha} = \{\tr\, \Y+ \sqrt{2\alpha\,\tr\,\Y^2}+\alpha\,\|\Y\| \le t_\alpha\}$. Then we can choose $t_\alpha > 0$ so that
$$
\frac{(1-5e^{-\alpha/2})\mathbb{P}^2(\mathcal{U}_\alpha)}{1+t_\alpha\sigma_n^2(\D)} \ge 1 - \frac{4k}n - \frac{8\log(3n)}n.
$$
\end{lemma}
\begin{proof}
A direct computation connects all three terms of $\mathcal{U}_{\alpha}$ to the spectral properties of the Gaussian $k \times n$ sketch matrix $\S$ as follows:
\begin{enumerate}
\item \quad $\tr\,\Y \leq \sigma_n^{-2}\cdot\tr\,(\S\S^\top)^{-1}\leq k\sigma_n^{-2}\cdot\|(\S\S^\top)^{-1}\|,$
    
\item \quad  $\sqrt{\tr\,\Y^2}\leq\sqrt k\sigma_n^{-2}\cdot\|(\S\S^\top)^{-1}\|,$ 

\item \quad $\|\Y\|\leq\sigma_n^{-2}\cdot\|(\S\S^\top)^{-1}\|.$
\end{enumerate}
\noindent
The spectral norm of the inverse $\|(\S\S^\top)^{-1}\| = \sigma^2_{\max}(\S^{-1}) = \sigma^{-2}_{\min}(\S)$ can be bounded by the following sharp concentration inequality for standard Gaussian matrices (e.g., \cite{davidson2001local}):
  \begin{align*}
    \mathbb{P}\big[\sigma_{\min}^2(\S)\leq (\sqrt{n} - \sqrt{k} - \sqrt{\alpha})^2\big] \leq \exp(-\alpha/2).
  \end{align*}
It implies that $\PP(\mathcal{U}_\alpha)\geq 1-e^{-\alpha/2}$ if we choose
\begin{align*}
    t_\alpha := \frac{\sigma_n^{-2}\cdot(k+\sqrt{2\alpha k} + \alpha)}{(\sqrt{n} - \sqrt{k} - \sqrt{\alpha})^2} \le \sigma_n^{-2}\cdot\Big(\frac{\sqrt k + \sqrt\alpha}{\sqrt n - \sqrt k - \sqrt\alpha}\Big)^2.
\end{align*}    
Now, with $\rho=(\frac{\sqrt k+\sqrt\alpha}{\sqrt n})^2$, $\delta = 7e^{-\alpha/2}$ and $\alpha=2\log n$, we have:
  \begin{align*}
    \epsilon &=
    1 - \frac{(1-5e^{-\alpha/2})\mathbb{P}^2(\mathcal{U}_\alpha)}{1+t_\alpha\sigma_n^2} 
    \leq 1 - \frac{1-7e^{-\alpha/2}}{1 + (\frac{\sqrt k + \sqrt\alpha}{\sqrt n - \sqrt k - \sqrt\alpha})^2}
    \\
    &=\frac1{(\frac1{\sqrt\rho}-1)^2+1} + \frac\delta{1+\frac1{(\frac1{\sqrt\rho} - 1)^2}} \leq 2\rho + \delta
    \\
    &= \frac{2(\sqrt k + \sqrt{2\log n})^2+7}n \le \frac{4k}n + \frac{8\log(3n)}n,
  \end{align*}  
  where we used the fact that for any $x>0$ we have:
  \begin{align*}
    \frac1{(\frac1{\sqrt x}-1)^2+1} =
 \frac{2x}{1 + (1-2\sqrt x)^2}\leq 2x.
  \end{align*}  
\end{proof}

\section{Proofs of Corollaries~\ref{cor:poly-decays} and \ref{c:flat-tailed}}\label{sec:app-b}

In this section, we show how Theorem \ref{t:gaussian} can be used to show improved convergence guarantees for sketch-and-project, given some additional knowledge about the spectral decay of the matrix. We start by stating a slightly modified version of the theorem that will be more convenient for the corollaries.
\begin{remark}\label{r:gaussian-variant}
As an immediate consequence of the proof of Theorem~\ref{t:gaussian}, one can also obtain the following bound for $k<(\sqrt n - 2)^2$:
$$
    \lambda_{\min}(\E[\P])\geq \frac{0.05}{1+C}\cdot \frac{k\sigma_{\min}^2(\A)}{\Err(\A,k-1)}\quad\text{where}\quad C = \Big(\frac{\sqrt k + 2}{\sqrt n - \sqrt k - 2}\Big)^2.
$$
This bound is less sharp than Theorem \ref{t:gaussian} for small sketch sizes $k$, but it is non-vacuous for a wider range of values of $n$ and $k$. In particular, we have $C\leq 100$  when $n\geq 100$ and $k\leq n/2$.
\end{remark}

\begin{proof}[Proof of Corollary~\ref{cor:poly-decays}] 
Note that when $k$ is bounded by an absolute constant, then the dependence of the convergence bounds on $k$ can be absorbed into the constant $C$, in which case the claim reduces to the results from prior work, e.g., Theorem 2 in \cite{rebrova2021block}. So, for the remainder of the proof, we will assume that $n\geq k\geq 100$.  For convenience, we will use the bound from Remark \ref{r:gaussian-variant}, and combine this with \eqref{eq:characterization}. The first claim follows immediately from Remark \ref{r:gaussian-variant} for sketch sizes $k\leq n/2$, since $\Err(\A,k)\leq\|\A\|_F^2$. To show it for larger sizes, it suffices to observe that $\lambda_{\min}(\E[\P])$ is non-decreasing as a function of $k$, so if we show a lower bound for $k=n/2$, then the same bound holds for $k>n/2$. The factor $2$ in the bound (coming from using a different $k$) can be absorbed into the constant.

We next prove the polynomial decay part of Corollary~\ref{cor:poly-decays}, by combining Remark~\ref{r:gaussian-variant} with Lemma~\ref{l:spectral1}. Here, again, if $k<4\beta$, then the dependence on $k$ can be absorbed into the constant $C$, so we assume that $k\geq 4\beta$. Let $p=k/\beta-1$ in the lemma. 
We can now bound the Randomized SVD error for sketch size $k-1$:
\begin{align*}
     \Err(\A,k-1)
     &\leq \frac{k-2}{p-1}\cdot\sum_{i\geq k-1-p}\sigma_i^2
     \leq \frac{k-2}{\frac k\beta-2}\cdot c\sigma_1^2\sum_{i\geq k-1-p}i^{-\beta}
     \\
     &\leq \beta\,\frac{k-2}{k-2\beta}\cdot \frac{c\sigma_1^2}{(\beta-1)(k-1-p)^{\beta-1}}
     \\
     &\leq \frac{k-2}{k-2\beta}\cdot
     \Big(\frac{\beta}{\beta-1}\Big)^\beta
     \frac{c\sigma_1^2}{k^{\beta-1}}
     \leq \frac{k-2}{k-2\beta}\cdot
     \frac{\beta}{\beta-1}\cdot \frac{ec\sigma_1^2}{k^{\beta-1}}.
\end{align*}
As a result, using Theorem \ref{t:gaussian} (Remark \ref{r:gaussian-variant}) and letting $\kappa(\A)=\sigma_{\max}(\A)/\sigma_{\min}(\A)$, we conclude that there is a constant $C>0$ such that:
\begin{align*}
\lambda_{\min}(\E[\P])\geq\Big(1-2\,\frac{\beta-1}{k-2}\Big)\Big(1-\frac1\beta\Big)\cdot\frac{k^\beta}{Cc\kappa^2(\A)}.
\end{align*}
So, since $\kappa^2(\A)\geq \|\A\|_F^2/\sigma_{\min}^2(\A)$ and $k\geq 4\beta$, we get the claim by adjusting the constant.
The result for exponential decay follows analogously, using Lemma \ref{l:spectral1} with $p=2$.
\end{proof}

\begin{proof}[Proof of Corollary~\ref{c:flat-tailed}] 
Here, we will again combine Remark~\ref{r:gaussian-variant} with Lemma~\ref{l:spectral1}. First, consider any $k\geq \max\{1.5r,10\}$. By the spectral assumption, matrix $\A$ satisfies the property that for any $i\geq 2k/3$, we have $\sigma_i^2(\A)\leq c\cdot \sigma_{\min}^2(\A)$. We will now apply Lemma \ref{l:spectral1} with $p=k/3-1$. In this case, we have:
\begin{align*}
\Err(\A,k-1)\leq \frac{k-2}{k/3-2}\cdot\sum_{i\geq 2k/3} \sigma_i^2 \leq 6\cdot \sum_{i\geq 2k/3}c\cdot \sigma_{\min}^2(\A)\leq 6cn\cdot \sigma_{\min}^2(\A).
\end{align*}
Thus, as long as $k$ and $n$ are in the range of values supported by Remark \ref{r:gaussian-variant}, we obtain the desired convergence rate, with some constant $C>0$:
\begin{align}
    \lambda_{\min}(\E[\P]) \geq \frac{k\sigma_{\min}^2(\A)}{C6cn\sigma_{\min}^2(A)}= \frac{k}{6Ccn}.\label{eq:flat}
\end{align}
Note that by making the constant $C_1$ in the statement of the claim sufficiently large, we can assume that $n\geq k\geq 250$, in which case we can ensure that $C$ in Remark \ref{r:gaussian-variant} is bounded by an absolute constant for any $k\leq3n/4$. Now, we consider two cases. Case~1: if $2r\geq n$, then the corollary's claim is trivially satisfied because $k\geq \max\{2r,C_1\}$ implies that $k\geq n$. Case~2: Suppose that $2r<n$. In this case, $\max\{1.5r,10\}\leq3n/4$, so there is a non-empty set of sketch sizes $\max\{1.5r,10\}\leq k \leq 3n/4$ that fall in the range of Remark \ref{r:gaussian-variant} and at the same time satisfy bound \eqref{eq:flat}. Having shown the desired bound for all $k$ up to $k=3n/4$, we can rely on the monotonicity of $\lambda_{\min}(\E[\P])$ as a function of $k$, just like in the proof of Corollary~\ref{cor:poly-decays}, to extend \eqref{eq:flat} to $k>3n/4$ (paying an extra factor of at most $4/3$ in the bound).
\end{proof}

\section{Sub-gaussian sketching: proof of Theorem \ref{t:sub-gaussian}}
\label{app:subgaus-proof}
\begin{proof}[Proof of Lemma~\ref{l:def-delta}]
Denote $\sigma_1^2\geq \sigma_2^2\geq ...\geq \sigma_n^2$ as the eigenvalues of $\Sigmab=\A^\top\A$. 
Note that $\P_{-i}$ is a rank $k-1$ projection, which means that $\A\P_{-i}$ can be viewed as a rank $k-1$ approximation of $\A$. Thus, its Frobenius norm approximation error $\|\A-\A\P_{-i}\|_F^2=\tr\Sigmab\Pperp{-i}$ has to be at least as large as the error of the best rank $k-1$ approximation, which is $\sum_{i\geq k}\sigma_i^2$. Thus,
\begin{equation*}\tr\Sigmab\Pperp{i}\geq \sum_{i\geq k}\sigma_i^2.\end{equation*}
Further, using that $\|\A\|_F^2 = \sum_{i}\sigma_i^2$, $\|\A\|^2 = \sigma_1^2$ and the stable rank $r =  \|\A\|_F^2/\|\A\|^2$, we continue lower estimating as
\begin{equation*}
\sum_{i\geq k}\sigma_i^2 \geq \sum_{i}\sigma_i^2 - \sigma_1^2k =  \|\A\|_F^2 - \|\A\|^2 k= \|\A\|^2(r-k)\geq \|\A\Pperp{i}\|^2(r-k),
\end{equation*}
where in the last step we used the fact that a projection can only decrease operator norm of a matrix. Putting the ingredients together, we got
\begin{equation*}
\frac{\tr\Sigmab\Pperp{i}}{\|\A\Pperp{i}\|^2} \ge r - k.
\end{equation*}
Note that up to this moment we did not employ any distributional assumptions on $\S$. Now,  due to $L$-Euclidean concentration of each row of the sketching matrix $\S$ and independence between $\x_i$ and $\P_{-i}$,

$$
    \PP(\neg \mathcal{E}_i)
     \leq
     \E\,2\exp\Big(-c\frac{\tr\Sigmab\Pperp{i}}{L^2\|\A\Pperp{i}\|^2}\Big)
        \leq 2\exp(-c(r-k)/L^2). 
$$
Then,
$
\PP(\neg \EE)\leq 2k\exp(-c(r-k)/L^2)\leq2\exp(-c'r/L^2).
$
\end{proof}

\begin{lemma}\label{lem:cond-symm} Under the notations and conditions of Theorem~\ref{t:sub-gaussian}, under normalization $\|\Sigmab\| = 1$, and using event $\mathcal{E}$ as defined in Lemma \ref{l:def-delta}, we have:
	$$
	\|\overbar\P^{-1/2}\E[\P]\overbar\P^{-1/2}-\I\| \le \|\overbar\P^{-1/2}\E_\EE[\P]\overbar\P^{-1/2}-\I\|
	+ \|\overbar\P^{-1}\|\cdot 2\cdot \PP(\neg\Ec)
	$$
	and
	$$
	\|\overbar\P^{-1}\| \le 1+r\kappa^2.
	$$
\end{lemma}
\begin{proof}
	By the law of total probability,
	$$
	\|\E_{\mathcal{E}}[\P] - \E[\P]\| = \|\E_{\mathcal{E}} [\P](I - \PP(\mathcal{E})) -  \E_{\neg\mathcal{E}} [\P]\PP(\neg\Ec)\| = \left\|\E_{\Ec}[\P]-\E_{\neg\Ec}[\P]\right\|\cdot\PP(\neg\Ec).
	$$
	Then, 
    \begin{align*}
\|\overbar\P^{-1/2}\E[\P]\overbar\P^{-1/2}-\I\|
   &\leq 
   \|\overbar\P^{-1/2}\E_\EE[\P]\overbar\P^{-1/2}-\I\|
   +\|\overbar\P^{-1/2}(\E_{\Ec}[\P]-\E[\P])\overbar\P^{-1/2}\|
   \\
   &\leq \|\overbar\P^{-1/2}\E_\EE[\P]\overbar\P^{-1/2}-\I\|
   +\|\overbar\P^{-1}\|\cdot \left\|\E_{\Ec}[\P]-\E_{\neg\Ec}[\P]\right\|\cdot\PP(\neg\Ec)
   \\
   &\leq \|\overbar\P^{-1/2}\E_\EE[\P]\overbar\P^{-1/2}-\I\|
   + \|\overbar\P^{-1}\|\cdot 2\cdot \PP(\neg\Ec).
\end{align*} Further,
$$\|\overbar\P^{-1}\|= \|\I + \gamma^{-1}\Sigmab^{-1}\| \le 1 + \kappa^2/\gamma \leq 1+r\kappa^2$$
 since $r=\tr(\Sigmab) = \|\A\|_F^2 \ge  \E\|\A\Pperp{k}\|_F^2$.
\end{proof}
A key technical step of the proof of Theorem~\ref{t:sub-gaussian} is the following convenient three-part split of the symmetric spectral difference between expected sketched matrix and its surrogate.

\begin{lemma}\label{lem:three-part-splitting}
Let $\A$ be a full rank $m\times n$ matrix  and let $\S$ be a $k\times m$ random sketching matrix with i.i.d.~rows having mean zero and identity covariance. Under the notations and conditions of Theorem~\ref{t:sub-gaussian}, for any positive probability event $\EE$,
\begin{align*}
\|\overbar\P^{-\frac{1}{2}}\E_\EE[\P]\overbar\P^{-\frac{1}{2}}-\I\| &\leq \left\|\Sigmab^{-\frac{1}{2}}\E_{\Ec}\big[ \Pperp{k} \x_k\x_k^\top-\Pperp{k}\Sigmab\big]\Sigmab^{-\frac{1}{2}}\right\| +  \left\|\Sigmab^{-\frac{1}{2}}\E_{\Ec}[\P_{\perp k}\Sigmab - \P_\perp\Sigmab]\Sigmab^{-\frac{1}{2}}\right\|\\
      &+ \left\|\E_{\Ec}\bigg[\left(\frac{\E[\x_k^\top\Pperp{k}\x_k]}{\x_k^\top\Pperp{k}\x_k}-1\right)\cdot
        \Sigmab^{-1/2}\Pperp{k} \x_k\x_k^\top \Sigmab^{-1/2}\bigg]\right\|.
\end{align*}
\end{lemma}
The proof of Lemma~\ref{lem:three-part-splitting} relies on the following version of the Sherman-Morrison rank-one update formula, which applies to the Moore-Penrose pseudoinverse.
\begin{lemma}[\cite{10.2307/2099767,precise-expressions}]\label{lem:rank-one-update}
    For $\X \in \mathbb R^{k \times n}$ with $k<n$, denote $\P = \X^\dagger \X$ and $\P_{-k} = \X_{-k}^\dagger \X_{-k}$, with $\X_{-i} \in \mathbb R^{(k-1) \times n}$ the matrix $\X$ without its i-th row $\x_i \in \mathbb R^n$. Then, as long as $\x_k^\top(\I-\P_{-k})\x_k\neq 0$, we have:
    \begin{align*}
      (\X^\top\X)^\dagger\x_k = \frac{(\I - \P_{-k}) \x_k}{\x_k^\top (\I - \P_{-k}) \x_k}, \quad \P - \P_{-k} = \frac{(\I - \P_{-k}) \x_k \x_k^\top (\I - \P_{-k}) }{\x_k^\top (\I - \P_{-k}) \x_k}.
    \end{align*}
\end{lemma}
\begin{proof}[Proof of Lemma~\ref{lem:three-part-splitting}]
Since for $\X =\S\A$,
$$
\E \P = \E (\X^\dagger\X) = \E (\X^\top\X)^\dagger(\X^\top\X) = \sum_{i=1}^k \E (\X^\top\X)^\dagger\x_i\x_i^\top = k \E (\X^\top\X)^\dagger\x_k \x_k^\top,
$$
we have by Lemma~\ref{lem:rank-one-update}
\begin{equation}\label{shm}
\E \P = k \E \left[\frac{(\I - \P_{-k}) \x_k\x_k^\top}{\x_k^\top (\I - \P_{-k}) \x_k}\right].
\end{equation} 
One can check directly that $\overbar\P^{-1} = (\Sigmab + \gamma^{-1} \I)\Sigmab^{-1}$. Therefore, also using that $\|\U\V\| \leq \|\V\U\|$ if $\U\V$ is symmetric,
\begin{align*}
\|\overbar\P^{-1/2}\E_\EE[\P]&\overbar\P^{-1/2}-\I\| 
= \left\|\overbar\P^{-1/2}\left[\E_\EE[\P]\overbar\P^{-1} - \I\right]\overbar\P^{1/2}\right\| \\
=&\left\|(\Sigmab+\gamma^{-1}\I)^{1/2}\Sigmab^{-1/2}(\E_\EE[\P] \gamma^{-1}\Sigmab^{-1} - \E_\EE[\P_\perp])
\Sigmab^{1/2}(\Sigmab+\gamma^{-1}\I)^{-1/2}\right\| \\
\leq&\left\|\Sigmab^{-1/2}\E_\EE[\P] \gamma^{-1}\Sigmab^{-1/2} - \Sigmab^{-1/2}\E_\EE[\P_\perp]\Sigmab^{1/2}\right\|  \\
\overset{\eqref{shm}}{=} &\left\|
\gamma^{-1}\Sigmab^{-1/2}
\E_\EE\left[\frac{k\Pperp{k} \x_k\x_k^\top}{\x_k^\top\Pperp{k}\x_k}\right]
\Sigmab^{-1/2}- \Sigmab^{-1/2}\E_\EE[\P_\perp\Sigmab]\Sigmab^{-1/2}\right\|.
\end{align*}
We have $k/\gamma = \E\|\A\Pperp{k}\|_F^2 = \E\tr(\Sigmab \Pperp{k})=\E[\x_k^\top\Pperp{k}\x_k]$. With $\xi := \x_k^\top\Pperp{k}\x_k$, we can write
\begin{align*}
\|\overbar\P^{-1/2}\E_\EE[\P]\overbar\P^{-1/2}-\I\| 
\leq &\left\|\E_{\Ec}\left[\frac{\E\xi}{\xi}\cdot
   \Sigmab^{-1/2}\Pperp{k}\x_k\x_k^\top\Sigmab^{-1/2}]\right]
   -\Sigmab^{-1/2}\E_{\Ec}[\P_\perp\Sigmab]\Sigmab^{-1/2}\right\|
   \\
   \leq &\left\|\E_{\Ec}\bigg[\left(\frac{\E\xi}{\xi}-1\right)\cdot
        \Sigmab^{-1/2}\Pperp{k} \x_k\x_k^\top \Sigmab^{-1/2}\bigg]\right\| \\
        &\quad\quad + \left\|\Sigmab^{-1/2}\E_{\Ec}\big[ \Pperp{k} \x_k\x_k^\top-\Pperp{k}\Sigmab\big]\Sigmab^{-1/2}\right\| \\
        &\quad\quad\quad\,\,\,\, +\left\|\Sigmab^{-1/2}\E_{\Ec}[\P_{\perp k}\Sigmab - \P_\perp\Sigmab]\Sigmab^{-1/2}\right\|.
   \end{align*}
\end{proof}

Another technical corollary from the rank-one update formula is as follows
\begin{lemma}\label{lem:subgaus-sp-ext}
Under the notations and conditions of Theorem~\ref{t:sub-gaussian}, in particular, assuming that $r \ge C_1k$ for some convenient constant $C_1 > 0$, and assuming that $\|\Sigmab\| = 1$, we have
\begin{equation*}
\Pperp{k}\Sigmab\Pperp{k}\preceq 2(\Sigmab+\P)
\end{equation*}
and
\begin{equation*}
\Sigmab^{-1/2}\preceq \sqrt{\gamma}\,\overbar\P^{-1/2}\preceq\overbar\P^{-1/2}.
\end{equation*}
\end{lemma}
\begin{proof}
Using that  $\|\Sigmab\|=1$, the spectrum of $\Pperp{k}\Sigmab\Pperp{k}$ is bounded as follows:
    \begin{align*}
    \Pperp{k}\Sigmab\Pperp{k} 
    &= (\I-\P_{-k})\Sigmab(\I-\P_{-k})
    \nonumber\\ &
    =\Sigmab - \P_{-k}\Sigmab - \Sigmab\P_{-k} + \P_{-k}\Sigmab\P_{-k}
    \nonumber\\ &
    \preceq \Sigmab + \Sigmab^2+\P_{-k}^2 + \P_{-k}^2 \preceq 2(\Sigmab+\P).
\end{align*}
For the second part, since $\overbar\P = \gamma \Sigmab (\gamma \Sigmab + \I)^{-1} \preceq \gamma \Sigmab $ and $\gamma = k/\E\|\A\Pperp{k}\|_F^2 \le k/(r-k) \le 1$ by the condition on $r$,
\begin{equation*}
\Sigmab^{-1/2}\preceq \sqrt{\gamma}\,\overbar\P^{-1/2}\preceq\overbar\P^{-1/2}.
\end{equation*}
\end{proof}
The last auxiliary lemma quantifies the variance of quadratic forms for vectors under Euclidean concentration:
\begin{lemma}\label{lem:concentration-sub-gaussian}
Let $\x$ be an $n$-dimensional isotropic random vector that satisfies Euclidean concentration with constant $L$. Then, for any $n\times n$ matrix $\B$, we have:
\begin{align*}
    \Var\big[\x^\top\B\x\big] \leq CL^2\|\B\|(L^2\|\B\| + \|\B\|_*).
\end{align*}
\end{lemma}
\begin{proof}
First, suppose that $\B=\A^\top\A$ for some matrix $\A$.
Note that from Euclidean concentration we have that $\E[(\|\A\x\|-\|\A\|_F)^p]\leq CL^p\|\A\|^p$ for $p=2,4$. Thus, we have 
\begin{align*}
    \Var\big[\x^\top\B\x\big]
    &=\E\big[(\|\A\x\|^2-\|\A\|_F^2)^2\big]
    \\
    &=\E\big[(\|\A\x\|-\|\A\|_F)^2(\|\A\x\|+\|\A\|_F)^2\big]
    \\
    &\leq 2\,\E\big[(\|\A\x\|-\|\A\|_F)^4\big] + 8\|\A\|_F^2\E\big[(\|\A\x\|-\|\A\|_F)^2\big]
    \\
    &\leq 2C L^4\|\A\|^4 + 8\|\A\|_F^2CL^2\|\A\|^2
    \\
    &=2CL^2\|\B\|(L^2\|\B\|+4\|\B\|_*),
\end{align*}
Adjusting the constants we obtain the claim for a symmetric positive semidefinite $\B$. Next, suppose that $\B$ is only symmetric. Then, we can decompose it as a difference of two positive semidefinite matrices: $\B=\B_+-\B_-$, and moreover we have:
\begin{align*}
    \Var[\x^\top\B\x] 
    &= \E\Big[\big(\x^\top\B_+\x-\tr(\B_+) - (\x^\top\B_-\x-\tr(\B_-)\big)^2\Big]
    \\
    &\leq 2\Var[\x^\top\B_+\x] + 2\Var[\x^\top\B_-\x]
    \\
    &\leq 2CL^2\|\B_+\|(L^2\|\B_+\|+\|\B_+\|_*) + 2CL^2\|\B_-\|(L^2\|\B_-\|+\|\B_-\|_*)
    \\
    &\leq 4CL^2\|\B\|(L^2\|\B\|+\|\B\|_*).
\end{align*}
Finally, for a general matrix $\B$, it suffices to observe that $\x^\top\B\x = \x^\top\bar\B\x$ for a symmetric matrix $\bar\B=(\B+\B^\top)/2$, which completes the proof.
\end{proof}

\section{Additional details for the experiments}\label{app:add-experiments}
Here, we provide additional details regarding the data generation process.

The artificial matrices have dimensions $5000 \times 150$, and are generated to have various spectral profiles. It is done in the following way. Initial SVD decomposition $\A = \U \Sigmab \V$ is constructed for a standard Gaussian matrix $\A$ with the rows normalized to have unit norm. Such matrix has almost linear spectral decay and well-conditioned smallest singular value (diagonal entries of $\Sigmab$ roughly span the segment $[6.8, 4.8]$). Based on that, we substitute $\Sigmab$ with artificial diagonal matrices having the following spectral decays. Models 'lin.01', 'lin.025' and 'lin.035' are the matrices $\U \bar\Sigmab \V$ with the entries of $\bar\Sigmab$ having linear decays $\sigma_i = 6.8-l\cdot i$ with $l = 0.01, 0.025$, and $0.035$ (which results in $\sigma_{\min}$ being $5.3$, $3.1$ and $1.6$, respectively). 
Two polynomial decay models 'poly1' and 'poly1.5' are imposed by the entries of $\bar\Sigmab$ defined by $\sigma_i = 6.8\,i^{-l}$, for $l = 1$ and $l = 1.5$ respectively. Finally, in 'step20' model, $\bar\Sigmab$ has its $20$ top singular values the same as in 'lin.01' model and the rest are as in 'poly1' model, and 'step37' model does the same transition from linear to polynomial decay at the $37$-th singular value. Note that all models are constructed to share the same largest singular value, so that the smallest singular value is the correct indicator of how well-conditioned the system~is.

We also consider three other data matrices: a $5000 \times 150$ random matrix with entries randomly generated i.i.d. N(0,1) and then normalized to have unit length rows ('gaus'); a $1280 \times 256$ sub-matrix of the USPS dataset \cite{chang2011libsvm} ('usps'); a $1500 \times 300$ sub-matrix of the w8a-X dataset \cite{chang2011libsvm} ('w8a'). We note that 'w8a' is not a full-rank dataset, so the convergence error is computed with respect to the least squares solution, which randomized Kaczmarz is known to converge to on rank-deficient systems \cite{zouzias2013randomized}. 

 \paragraph{Additional experiments on real data} 
 
\begin{figure}
\centering
\includegraphics[width=\linewidth]{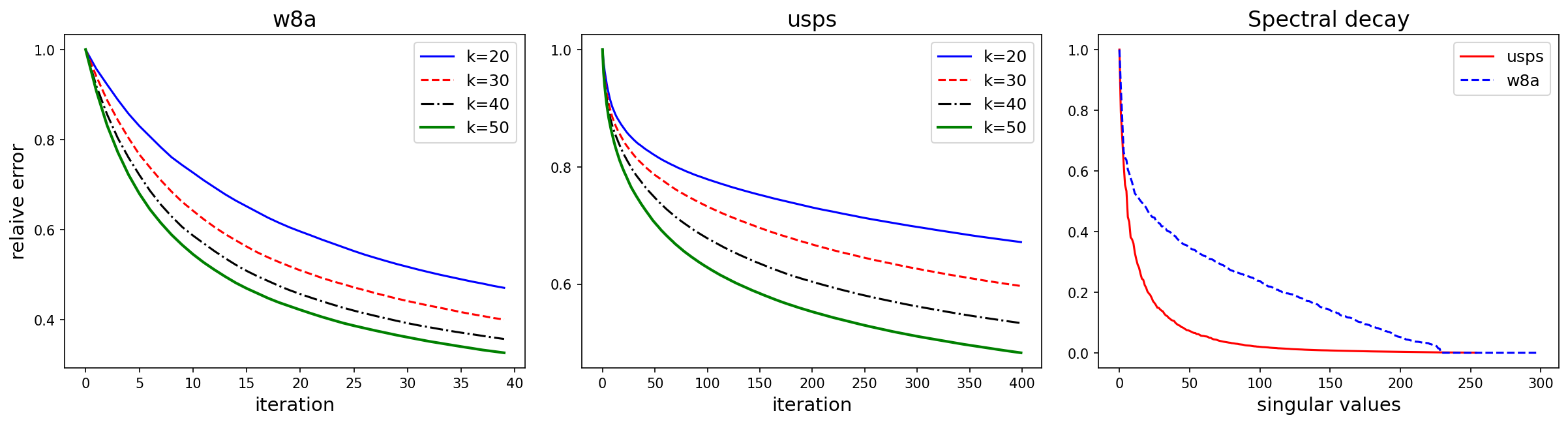}
\vspace{-2mm}
\caption{Per iteration convergence for two models based on real-world data matrices $\A$ (left and middle) and their spectral profiles (right). Sketch size is equal to $k$. The real-world dataset with slower (more ``linear") spectral decay ('w8a') shows less convergence improvement with bigger sketch sizes, compared to the dataset with a faster spectral decay ('usps').}
\label{fig-real-world-2}
\end{figure}
In Figure~\ref{fig-real-world-2}, we show how sketch size influences the convergence behavior of sketched Kaczmarz on real-world datasets 'w8a' and 'usps', by plotting relative error after a certain number of iterations. This figure complements the results on artificial matrices, shown in Figure~\ref{fig-real-world-1}. In both figures, the relative error after $t$ steps is computed as
 $\|\x_t - \x_*\|/\|\x_*\|$. Our initial guess $\x_0$ is the all zeros vector. 

 Similar effect is observed on real-world datasets (Figure~\ref{fig-real-world-2}): the spectral decay of 'w8a' is flatter than the spectrum of 'usps', and increasing the sketch size $k$ leads to less significant improvement of convergence speed.

\end{document}